\documentclass[aop]{imsart}

\usepackage{amsthm,amsmath,amssymb}

\usepackage{xcolor}

\RequirePackage[numbers]{natbib}
\RequirePackage[colorlinks,citecolor=blue,urlcolor=blue]{hyperref}

\startlocaldefs
\numberwithin{equation}{section}
\theoremstyle{plain}
\newtheorem{thm}{Theorem}[section]
\newtheorem{lem}{Lemma}[section]
\newtheorem{dfn}{Definition}[section]
\newtheorem{prp}{Proposition}[section]
\newtheorem{rmk}{Remark}[section]
\newtheorem{cor}{Corollary}[section]
\endlocaldefs

\begin{document}
\newcommand{\Q}{\mathbb{Q}}
\newcommand{\R}{\mathbb{R}}
\newcommand{\I}{\mathbb{I}}
\newcommand{\Z}{\mathbb{Z}}
\newcommand{\N}{\mathbb{N}}
\newcommand{\F}{\mathcal{F}}
\newcommand{\p}{\mathbf{P}}
\newcommand{\B}{\mathcal{B}}
\newcommand{\h}{\mathcal{H}}
\newcommand{\M}{\mathcal{M}}
\newcommand{\E}{\mathbf{E}}
\newcommand{\Cf}{\mathrm{C}}
\newcommand{\En}{\mathrm{E}^n}
\newcommand{\eps}{\varepsilon}
\newcommand{\supp}{\mathrm{supp\,}}
\newcommand{\law}{\mathrm{Law}}
\newcommand{\Leb}{\mathrm{Leb}}
\newcommand{\sgn}{\mathrm{sgn\,}}
\newcommand{\inter}{\mathrm{int\,}}

\newcommand{\col}{}
\newcommand{\cob}{}
\newcommand{\cog}{}

\begin{frontmatter}

\begin{aug}
\title{A system of coalescing heavy diffusion particles on the real line}
\runtitle{A system of heavy diffusion particles}

\author{\fnms{Vitalii} \snm{Konarovskyi}\corref{}\ead[label=e1]{konarovskiy@gmail.com}}
\address{Max Planck Institut f\"{u}r Mathematik\\
in den Naturwissenschaften,\\
Inselstraße 22,\\
04103 Leipzig,\\
Germany\\
and\\
Department of Mathematics and Informatics,\\
Yuriy Fedkovych Chernivtsi
National University,\\
Kotsubinsky Str. 2,\\
58012 Chernivtsi,\\
Ukraine\\
\printead{e1}}
\affiliation{Max Planck Institut f\"{u}r Mathematik in den Naturwissenschaften and\\
Yuriy Fedkovych Chernivtsi National University}

\runauthor{V. Konarovskyi}

\end{aug}

\begin{abstract}
We construct a modified Arratia flow with mass and energy conservation. We suppose that particles have a mass obeying the conservation law, and their diffusion is inversely proportional to the mass. Our main result asserts that such a system exists under the assumption of the uniform mass distribution on an interval at the starting moment. We introduce a stochastic integral with respect to such a flow and obtain the total local time as the density of the occupation measure for all particles.
\end{abstract}

\begin{keyword}[class=MSC]
\kwd[Primary ]{60K35}
\kwd{82B21}
\kwd[; secondary ]{60H05}
\kwd{60J55}
\end{keyword}

\begin{keyword}
\kwd{Interacting particle system}
\kwd{Arratia flow}
\kwd{coalescing}
\kwd{stochastic integral with respect to flow}
\kwd{Ito formula}
\kwd{local time}
\end{keyword}

\end{frontmatter}


\section{Introduction and statement of main results}

The paper is devoted to a model of interacting diffusion particles on the real line. {\cog Intuitively the new model can be understood as follows.} The particles start from all points of a fixed interval (for convenience, we consider the interval $[0,1]$), move independently up to the {\cog meeting time then coalesce and stay together.} {\cog Each particle carries a mass and when two particles coalesce the resulting particle carries the mass of both. This implies the mass conservation, as for example, in~\cite{Smoluchowski:1916,Lang:1980,Weinan:1996}.} In addition, we suppose that the diffusion {\cog rate of each} particle is inversely proportional to its mass. {\cog This is a new feature of our model which is not present in the classical ones. While this new mechanism makes the model physically more realistic it also makes its investigation more complicated. It should be noted that the dependence of the diffusion on the mass distinguishes our model from those that are actively investigated such as the well-known Arratia flow~\cite{Arratia:1979,Arratia:1981,Dorogovtsev:2004,Norris:2015}, where every subsystem can be described as a separate system. This fact facilitates the study of fine properties of the Arratia flow as in~\cite{Le_Jan:2004,Evans:2013,Dorogovtsev:2006,Dorogovtsev:2007:en,Dorogovtsev:2005,Malovichko:2009,Chernega:2012,Shamov:2011,Vovchanskii:2012}.}

Systems of interacting particles with a mass or measure-valued processes corresponding to them arise in statistical mechanics, where particles are interpreted as molecules of gas or liquid, in genetics, where the phase space is a space of possible genotypes and the mass of a particle corresponds to the share of individuals of a population that have some genotype, in hydrodynamics and cosmology, where the mass is interpreted as a naturally physical mass of molecules of liquid or gas, in turbulence theory, where particles are interpreted as curls and the mass corresponds to circulation. Such models of particles with masses were studied by M.~V.~Smoluchowski~\cite{Smoluchowski:1916}, R.~Lang~\cite{Lang:1980}, D.~A.~Dawson~\cite{Dawson:2001,Dawson:2004}, H.~Wang~\cite{Wang:1997,Wang:1998}, W.~H.~Fleming~\cite{Fleming:1979}, L.~G.~Gorostiza~\cite{Gorostiza:1990,Gorostiza:1991}, J.~R.~Norris~\cite{Norris:1999}, Ya.~G.~Sinai~\cite{Weinan:1996}, A.~A.~Dorogovtsev~\cite{Kotelenez:1997,Dorogovtsev:2002,Dorogovtsev:2003}, M.~P.~Karlikova~\cite{Karlikova:2005}, Konarovskyi~V.~V.~\cite{Konarovskiy:2010:TVP:en,Konarovskiy:2010:UMJ:en,Konarovskiy:2011:TSP,Konarovskyi:2013:COSA,Konarovskyi:2014:TSP} and others. {\cog  In some models, such as those studied in~\cite{Smoluchowski:1916,Lang:1980,Weinan:1996} the mass influences the motion of the particles, but in many other cases this does not happen~\cite{Dawson:2001,Dawson:2004,Wang:1997,Wang:1998}.}

{\cog The author has already investigated in~\cite{Konarovskiy:2010:TVP:en,Konarovskiy:2010:UMJ:en,Konarovskyi:2013:COSA} simpler discrete models where the diffusion rate is inversely proportional to the mass.} {\col In these papers a countable system of particles started with positive mass and the mass of the whole system was infinite. The difference in this work is the assumption that the particles start from all points of an interval with ``infinitesimal mass''.} {\cog We prove existence for such model of particles, which is a delicate issue due to the fact that particles start with zero mass and therefore infinite diffusion rate. The main reason for existence is the fact that particles coalesce immediately to a finite set of points. Consequently, the particles have a finite diffusion rate at any positive time.} Since we must simultaneously consider an uncountable number of particles, another important question is the method of defining the system of processes describing the evolution of the particles. To do this, we use a martingale approach. We construct a continuum of martingales that satisfy certain properties characterizing our model. {\col Let $\Leb$ denote the Lebesgue measure on $[0,1]$.} {\cog The following theorem is the main result of the paper.}

\begin{thm}\label{theorem_main_result}
There exists an random element $\{y(u,t),\ u\in[0,1],\ t\in[0,T]\}$ in the Skorohod space $D([0,1],C[0,T])$ such that
\begin{enumerate}
\item[(C1)] for all $u\in[0,1]$, the process $y(u,\cdot)$ is a continuous square integrable martingale with respect to the filtration
\begin{equation}\label{f_filtration_y}
\mathcal{F}_t=\sigma(y(u,s),\ u\in[0,1],\ s\leq t),\quad t\in[0,T];
\end{equation}

\item[(C2)] for all $u\in[0,1]$, $y(u,0)=u$;

\item[(C3)] for all $u<v$ from $[0,1]$ and $t\in[0,T]$, $y(u,t)\leq y(v,t)$;

\item[(C4)] for all $u\in[0,1]$, the {\col quadratic variation} has the form
$$
\langle y(u,\cdot)\rangle_t=\int_0^t\frac{ds}{m(u,s)},
$$
where $m(u,t)=\Leb\{v:\ \exists s\leq t\ y(v,s)=y(u,s)\}$, $t\in[0,T]$;

\item[(C5)]  for all $u,v\in[0,1]$ and $t\in[0,T]$,
$$
\langle y(u,\cdot),y(v,\cdot)\rangle_{t\wedge\tau_{u,v}}=0,
$$
where $\tau_{u,v}=\inf\{t:\ y(u,t)=y(v,t)\}\wedge T$.
\end{enumerate}
\end{thm}

Here $y(u,t)$ will be interpreted as the position of the particle starting from $u$ at a time $t$. Let us briefly explain conditions $(C1)$--$(C5)$. Conditions $(C1)$ and $(C2)$ are responsible for the fact that we have a set of {\cog diffusing} particles starting from all points of $[0,1]$. {\col Condition $(C3)$ reflects the coalescing behaviour of the particles. $(C4)$ and $(C5)$ give the diffusion rate, and a sort of ``independence'' of the motions up to their collision time.} It should be noted, that since the diffusion of every particle depends on how many particles coalesce to it, we cannot talk about the usual independence of the movement of particles up to the moment of meeting {\col as for the Arratia flow (in our case the motion depends on the mass).} {\cog However, in between two collision times particles move independently.}

{\cog At the moment, we do not know whether $(C1)$--$(C5)$ imply uniqueness. This remains an interesting open problem. It seems reasonable to conjecture that uniqueness holds. However, as it turns out, this approach is very useful for obtaining qualitative properties of the system.} For example, using $(C1)$--$(C5)$ in the paper we construct the stochastic integral with respect to the flow denoted by
$$
\int_0^1\int_0^t\varphi(y(u,s))dy(u,s)du
$$
(see Proposition~\ref{prop_fundam}), which is different from the integral with respect to the Arratia flow introduced by A.~A.~Dorogovtsev~\cite{Dorogovtsev:2007:en}. Namely the integral with respect to the Arratia flow is the sum of integrals over all pieces of trajectories up to the moment of coalescing. In our case, we integrate, roughly speaking, over ``the measure'' $d_sy(u,s)du$. {\col Using this integral, we obtain the analog of Ito's formula for functionals of the form $\int_0^1\varphi(y(u,t))du$.} Next, {\cog using this analog of Ito's formula, we establish the existence of the local time $\{L(a,t),\ a\in\R,\ t\in[0,T]\}$. This is the density} of the occupation measure
$$
\mu(A)=\int_0^1\int_0^{\tau(u)\wedge t}\I_A(y(u,s))ds
$$
{\cog (we refer to~Chapter~7~\cite{Dorogovtsev:2007:en} for the precise meaning of this object).}
{\col We also establish a Tanaka formula for the local time:}
\begin{align*}
L(a,t)&=\int_0^1(y(u,t)-a)^+du-\int_0^1(u-a)^+du\\
&-\int_0^1\int_0^t\I_{(a,+\infty)}(y(u,s))dy(u,s)du.
\end{align*}
Here the definition of the local time exactly coincides with one which was introduced for the Arratia flow in~\cite{Chernega:2012}.

Let us briefly describe the idea of the proof of Theorem~\ref{theorem_main_result} and the structure of the article. To build a system of particles starting from all points of the interval, we use the thermodynamic limit {\cog as in}~\cite{Ruelle:1969}, that is, we approximate our system by the system of particles starting from $\frac{k}{n}$, $k=1,\ldots,n$, with the mass $\frac{1}{n}$. {\cog We construct this approximate system in Proposition~\ref{prop_finite_syst}}. Next we pass to the limit as the number of particles tends to infinity. {\col Since the diffusion rate of the particles at the start} tends to infinity, we will pass to the limit in two steps. First, in Section~\ref{section_tightness} we show {\col that a sequence approximating} the continuum particle system is tight in the space $D([0,1],C(0,T])$, and {\col hence a subsequence is weakly convergent} to an element $\{y(u,t),\ u\in[0,1],\ t\in(0,T]\}$ in $D([0,1],C(0,T])$. In order to show this, we use some ideas of the paper~\cite{Piterbarg:1998}, in which the author checks the convergence rescaling homeomorphic isotropic stochastic flows {\col to the Arratia flow.} Next, we extend $\{y(u,t),\ u\in[0,1],\ t\in(0,T]\}$ to $t=0$ in Section~\ref{section_extension}. To do this, first we establish a property of the sequence (Proposition~\ref{prop_tight_xi}) and using it we show that
$$
\int_0^1\varphi(y(u,t))du\to\int_0^1\varphi(u)du\ \ \mbox{in probability},\ \ t\to 0.
$$
{\cog Thanks to this} property, the monotonicity of $y(u,\cdot),\ u\in[0,1]$, and the fact that $y(u,\cdot)$ is a continuous martingale for each $u$ {\cob (see Lemmas~\ref{lemma_converg_a_s} and~\ref{lemma_cont})}, we establish the possibility of extending $\{y(u,t),\ u\in[0,1],\ t\in(0,T]\}$ to the whole interval $[0,T]$. Section~\ref{section_C4_C5} is devoted to checking of conditions $(C4)$ and $(C5)$. {\cog Using conditions $(C2)$--$(C5)$, we derive in Section~\ref{section_C1} some estimates for the expectation of the diffusion rate of the} particles and show that $y(u,\cdot)$ is a continuous square integrable martingale, for each $u$. In Section~\ref{section_Ito_formula} we introduce the definition of a stochastic integral with respect to the flow of heavy diffusion particles, as a limit of partial sums and obtain an analog of Ito's formula. In Section~\ref{section_local_time}, {\cog we establish existence of the local time using the Ito formula. From Section~\ref{section_C1} on, we do not assume that the set of processes $\{y(u,t),\ u\in[0,1],\ t\in[0,T]\}$ is the limit of a finite systems, but we only assume that it is some process satisfying $(C1)$--$(C5)$.}

\section{Construction of the system}

\subsection{A finite system of particles}\label{section_finite_system}

In this section we construct a system of processes that describes an evolution of diffusion particles on the real line. We suppose that particles start from a finite number of points, move independently up to the moment of the meeting and coalesce, {\col and change their diffusion rates accordingly.} Since we approximate a system of particles starting from all points of the interval $[0,1]$ by a finite system, it is enough to consider the case where particles start from the points $\frac{k}{n}$, $k=1,\ldots,n$, with the mass $\frac{1}{n}$. So, let $n\in\mathbb{N}$ be fixed. Denote $[n]=\{1,2,\ldots,n\}$.

\begin{dfn}
A set $\pi=\{\pi_1,\ldots,\pi_p\}$ of non-intersecting subsets of $[n]$ is called {\col an order partition} of $[n]$ if
\begin{enumerate}
\item[1)] $\bigcup\limits_{i=1}^p\pi_i=[n]$;

\item[2)] if $l,k\in\pi_i$ and $l<j<k$ then $j\in\pi_i$, for all $i\in[p]$.
\end{enumerate}
\end{dfn}

The set of all order {\col partitions} of $[n]$ is denoted by~$\Pi^n$.

Every element $\pi=\{\pi_1,\ldots,\pi_p\}\in\Pi^n$ generates an equivalence relation between $[n]$ elements. We write $i\sim_{\pi}j$ provided there exists a number $k$ such that $i,j\in\pi_k$. Denote the equivalence class that contains the element $i\in[n]$ by $\widehat{i}_{\pi}$, i.e.
$$
\widehat{i}_{\pi}=\{j\in[n]:\ j\sim_{\pi}i\}.
$$

Using a system of independent Wiener processes $\{w_k(t),\ t\in[0,T],\ k\in[n]\}$ we construct the required system. Denote
$$
\tau^0=0,\quad\pi^0=\{\{k\},\ k\in[n]\}\in\Pi^n
$$
and
$$
w_k^0(t)=\frac{k}{n}+\frac{1}{\sqrt{n}}w_k(t),\quad t\in[0,T],\ \ k\in[n].
$$
Define by induction for $p\in[n-1]$
$$
\tau^p=\inf\{t>\tau^{p-1}:\ \exists i,j\in[n],\ i\not\sim_{\pi^{p-1}}j,\ w^{p-1}_i(t)=w^{p-1}_j(t)\}\wedge T.
$$
Take $\pi^p\in\Pi^n$ such that
$$
i\sim_{\pi^p}j\Leftrightarrow w^{p-1}_i(\tau^p)=w^{p-1}_j(\tau^p)
$$
and set for $k\in[n]$
$$
w_k^p(t)=\left\{\begin{array}{ll}
w_k^{p-1}(t), & t\leq\tau^p,\\
\sum_{i\in\widehat{k}_{\pi^p}}\frac{w_i^0(t)}{|\widehat{k}_{\pi^p}|}& t>\tau^p.
\end{array}\right.
$$

Denote for convenience $x_k^n(t)=w_k^{n-1}(t),\ t\in[0,1],\ k\in[n]$.

\begin{prp}\label{prop_finite_syst}
The set of the processes $\{x_k^n(t),\ k\in[n],\ t\in[0,T]\}$ satisfies the following conditions
\begin{enumerate}
\item[(F1)] for each $k\in[n]$, $x_k^n(\cdot)$ is a continuous square integrable martingale with respect to the filtration
$$
\F_t^n=\sigma(x_l^n(s),\ s\leq t,\ l\in[n]);
$$

\item[(F2)] for all $k\in[n]$, $x_k^n(0)=\frac{k}{n}$;

\item[(F3)] for all $k<l$ and $t\in[0,T]$, $x_k^n(t)\leq x_l^n(t)$;

\item[(F4)] for all $k\in[n]$, the quadratic variation has the form
$$
\langle x_k^n(\cdot)\rangle_t=\int_0^t\frac{ds}{m_k^n(s)},
$$
where $m_k^n(t)=\frac{1}{n}|\{j:\ \exists s\leq t\ x_j^n(s)=x_k^n(s)\}|$, $t\in[0,T]$;

\item[(F5)] for all $k,l\in[n]$ and $t\in[0,T]$,
$$
\langle x_k^n(\cdot),x_l^n(\cdot)\rangle_{t\wedge\tau_{k,l}}=0,
$$
where $\tau_{k,l}^n=\inf\{t:\ x_k^n(t)=x_l^n(t)\}\wedge T$.
\end{enumerate}
\end{prp}

The proof of this proposition can be easily derived from the above construction of the processes $x_k^n(\cdot),\ k\in[n]$. In~\cite{Konarovskiy:2010:TVP:en} the author proved that conditions $(F1)$--$(F5)$ of Proposition~\ref{prop_finite_syst} uniquely determine the distribution of $(x_k^n(\cdot))_{k\in[n]}$ in the space of continuous functions from $[0,T]$ to $\R^n$. In other words, if a set of processes $\{\xi_k(t),\ k\in[n],\ t\in[0,T]\}$ satisfies $(F1)$--$(F5)$ then the distributions of $(x_k^n(\cdot))_{k\in[n]}$ and $(\xi_k(\cdot))_{k\in[n]}$ coincide.

Let us prove a property of the constructed system which will be used.

\begin{lem}\label{lemma_expectation}
{\col $(i)$ If $C=\sqrt{T}+1$, then for all $n\in\N$, $k\in[n]$} and $t\in[0,T]$
$$
\E |x_k^n(t)|\leq C.
$$
{\col $(ii)$ For each $0<p_1<p_2<1$ there exists a constant $C(p_1,p_2)$ such that
$$
\E\max_{t\in[0,T]}(x_k^n(t))^2\leq C(p_1,p_2),
$$
for all $n\geq\frac{1}{p_1}$ and $k\in[n]$ satisfying $\frac{k}{n}\in(p_1,p_2)$.}
\end{lem}

\begin{proof}
Consider the process
$$
\eta_n(t)=\frac{1}{n}\sum_{k=1}^nx_k^n(t),\quad t\in[0,T].
$$
Note that by condition $(F1)$, $\eta_n(\cdot)$ is a continuous square integrable martingale. Using Ito's formula and condition $(F4)$ we obtain
$$
\eta_n^2(t)=\mbox{mart.}+\frac{1}{n^2}\sum_{k=1}^n\sum_{l=1}^n\int_0^t\frac{\I_{\{\tau^n_{k,l}\leq s\}}ds}{m_k^n(s)}=\mbox{mart.}+t.
$$
So, $\eta_n(\cdot)$ is a continuous square integrable martingale with the quadratic variation $\langle\eta_n(\cdot)\rangle_t=t$, $t\in[0,T]$. {\col By the martingale characterization} of Brownian motion (see Theorem~2.6.1~\cite{Watanabe:1981:en}), $\eta_n(\cdot)$ is a Wiener process. {\col To bound the expectation of $|x_k^n(t)|$, write}
\begin{align*}
\E |x_k^n(t)| & \leq\E |x_k^n(t)-\eta_n(t)|+\E|\eta_n(t)|\\
              &\leq\E(x_n^n(t)-x_1^n(t))+\sqrt{T}=\sqrt{T}+1.
\end{align*}
The latter inequality follows from conditions $(F1)$ and $(F2)$.

{\col Next we prove the second part of the lemma. Let $n\geq\frac{1}{p_1}$ be fixed. Set
$$
 A_1=\left\{l\in[n]:\ \frac{l}{n}\leq p_1\right\},\quad
 A_2=\left\{l\in[n]:\ \frac{l}{n}\geq p_2\right\}.
$$
Note that $A_1$ and $A_2$ is non-empty, by the choice of $n$. By $(F1)$, the processes
$$
M_i(t)=\frac{1}{|A_i|}\sum_{l\in A_i}x_l^n(t),\quad t\in[0,T],\ \ i=1,2,
$$
are continuous square integrable martingales. Using $(F3)$, we have for all $t\in[0,T]$
$$
M_1(t)\leq x_k^n(t)\leq M_2(t),\quad\mbox{if }\ \frac{k}{n}\in(p_1,p_2).
$$
Thus
\begin{align*}
 \E\max_{t\in[0,T]}(x_k^n(t))^2&\leq\E\max_{t\in[0,T]}(M_1^2(t)\vee M_2^2(t))\\
 &\leq\E\max_{t\in[0,T]}M_1^2(t)+ \E\max_{t\in[0,T]}M_2^2(t).
\end{align*}
Hence by The Burkholder-Davis-Gundy inequality
$$
\E\max_{t\in[0,T]}M_i^2(t)\leq \E\langle M_i\rangle_T,\quad i=1,2.
$$
Let us estimate the quadratic variation of $M_i$. By conditions $(F4)$ and $(F5)$,
$$
\langle M_i\rangle_T=\frac{1}{|A_i|^2}\sum_{l\in A_i}\sum_{j\in A_i}\int_0^T\frac{\I_{\{\tau_{j,l}^n\leq s\}}}{m_l^n(s)}ds.
$$
Using the relation $\sum_{j\in A_i}\I_{\{\tau_{j,l}^n\leq s\}}=|\{j:\ \exists r\leq s\ x_j^n(r)=x_l^n(r)\}\cap A_i|\leq n m_l^n(s)$, we obtain
\begin{align*}
\langle M_i\rangle_T\leq \frac{nT}{|A_i|^2} \sum_{l\in A_i}1=\frac{nT}{|A_i|},
\end{align*}
where $|A_1|=\lfloor np_1\rfloor$ and $|A_2|=\lfloor n(1-p_2)\rfloor+1$. It finishes the proof of the lemma.
}
\end{proof}

\subsection{Tightness in Skorohod space $D([0,1],C(0,T])$}\label{section_tightness}

Let $C[a,b]$ denote the metric space of continuous functions from $[a,b]$ to $\R$ with the uniform distance, and $C(0,T]$ denote the metric space of continuous functions from $(0,T]$ to $\R$ with the metric generated by the uniform convergence on compact subsets of $(0,T]$. Denote by $D([0,1],E)$ the space of right continuous functions from $[0,1]$ to a metric space $E$ with left limits, equipped with the standard Skorohod topology.

Let us set
$$
y_n(u,\cdot)=\begin{cases}
x_k^n(\cdot),\quad \frac{k-1}{n}\leq u<\frac{k}{n},\ \ k\in[n],\\
x_n^n(\cdot),\quad u=1,
\end{cases}
$$
and note that $y_n=\{y_n(u,t),\ u\in[0,1],\ t\in[0,T]\}$ is a random element of the space $D([0,1],C[0,T])$. We are going to show that the sequence $\{y_n\}_{n\geq 1}$ is tight. But from Condition $(F4)$, we can see that for large enough $n$ the mass of each particle is small for small time. It means that the fluctuations of the particles grow, so we cannot talk about tightness on the whole time interval $[0,T]$. For this reason, first we consider an evolution of the particles on the time interval $[\eps,T]$, where $\eps>0$, and using the fact that the particles coalesce quickly we prove the tightness of our system in $D([0,1],C[\eps,T])$. Then we conclude that the tightness in $D([0,1],C(0,T])$ holds.

\begin{prp}\label{prop_tightness_in_D_eps}
For all $\eps>0$ the sequence $\{y_n(u,t),\ u\in[0,1],\ t\in[\eps,T]\}_{n\geq 1}$ is tight in $D([0,1],C[\eps,T])$.
\end{prp}

First we prove several auxiliary lemmas.

\begin{lem}\label{lemma_three_points}
For all $n\in\N$, $u\in[0,2]$, $h\in[0,u]$ and $\lambda>0${\cob
$$
\p\{\|y_n(u+h,\cdot)-y_n(u,\cdot)\|\geq\lambda,\ \|y_n(u,\cdot)-y_n(u-h,\cdot)\|\geq\lambda\}\leq\frac{9h^2}{\lambda^2}.
$$  }
Here $y_n(u,\cdot)=y_n(1,\cdot)$, $u\in[1,2]$, and $\|\cdot\|$ is the uniform norm on $[0,T]$.
\end{lem}

\begin{proof}
Let $(\F_t^{y_n})_{t\in[0,T]}$ be the filtration generated by $y_n$, i.e.
\begin{equation}\label{f_filtration}
\F_t^{y_n}=\sigma(y_n(u,s),\ s\leq t,\ u\in[0,1]),\quad t\in[0,T].
\end{equation}

Consider the $(\F_t^{y_n})$-stopping times
\begin{align*}
\sigma^+&=\inf\{t:\ y_n(u+h,t)-y_n(u,t)\geq\lambda\}\wedge T,\\
\sigma^-&=\inf\{t:\ y_n(u,t)-y_n(u-h,t)\geq\lambda\}\wedge T,\\
\tau&=\inf\{t:\ y_n(u+h,t)-y_n(u,t)=0\ \mbox{or}\ y_n(u,t)-y_n(u-h,t)=0\}\wedge T
\end{align*}
and the process
\begin{align*}
M(t)=&(y_n(u+h,t\wedge\sigma^+)-y_n(u,t\wedge\sigma^+))\\
\cdot &(y_n(u,t\wedge\sigma^-)-y_n(u-h,t\wedge\sigma^-)),\quad t\in[0,T].
\end{align*}
We show that $M(\cdot)$ is a supermartingale. For this purpose we calculate the joint quadratic variation of $y_n(u_1,\cdot\wedge\sigma^+)$ and $y_n(u_2,\cdot\wedge\sigma^-)$, $u_1=u+h$ and $u$, $u_2=u$ and $u-h$.
\begin{align*}
\langle y_n(u+h,&\cdot\wedge\sigma^+),y_n(u,\cdot\wedge\sigma^-)\rangle_{t\wedge\tau}\\
&=\left\langle\int_0^{\cdot}\I_{\{s\leq\sigma^+\}}dy_n(u+h,s),\int_0^{\cdot} \I_{\{s\leq\sigma^-\}}dy_n(u,s)\right\rangle_{t\wedge\tau}\\
&=\int_0^{t\wedge\tau}\I_{\{s\leq\sigma^+\wedge\sigma^-\}}d\langle y_n(u+h,\cdot),y_n(u,\cdot)\rangle_s=0,
\end{align*}
since $\langle y_n(u+h,\cdot),y_n(u,\cdot)\rangle_t=0$, for all $t\leq\tau$. Similarly,
\begin{align*}
\langle y_n(u+h,\cdot\wedge\sigma^+),y_n(u-h,\cdot\wedge\sigma^-)\rangle_{t\wedge\tau}&=0,\\
\langle y_n(u,\cdot\wedge\sigma^+),y_n(u-h,\cdot\wedge\sigma^-)\rangle_{t\wedge\tau}&=0
\end{align*}
and
$$
\langle y_n(u,\cdot\wedge\sigma^+),y_n(u,\cdot\wedge\sigma^-)\rangle_{t\wedge\tau}= \int_0^{t\wedge\tau}\I_{\{s\leq\sigma^+\wedge\sigma^-\}}d\langle y_n(u,\cdot)\rangle_s=A(t).
$$

Since $y_n(u,\cdot\wedge\sigma^+)y_n(u,\cdot\wedge\sigma^-)-A(\cdot)$ is a martingale and the process $A(\cdot)$ does not decrease, $y_n(u,\cdot\wedge\sigma^+)y_n(u,\cdot\wedge\sigma^-)$ is a submartingale.
Write
\begin{align*}
M(t)&=M(t\wedge\tau)=y_n(u+h,t\wedge\sigma^+\wedge\tau)y_n(u,t\wedge\sigma^-\wedge\tau)\\
&-y_n(u+h,t\wedge\sigma^+\wedge\tau)y_n(u-h,t\wedge\sigma^-\wedge\tau)\\
&+y_n(u,t\wedge\sigma^+\wedge\tau)y_n(u-h,t\wedge\sigma^-\wedge\tau)\\
&-y_n(u,t\wedge\sigma^+\wedge\tau)y_n(u,t\wedge\sigma^-\wedge\tau).
\end{align*}
The first three terms are martingales and the last term is a submartingale, so $M(\cdot)$ is a supermartingale.

Note that $M(T)\geq\lambda^2\I_{\{\sigma^+\vee\sigma^-<T\}}$. Hence
\begin{align*}
\p\{\|y_n(u&+h,\cdot)-y_n(u,\cdot)\|\geq\lambda,\ \|y_n(u,\cdot)-y_n(u-h,\cdot)\|\geq\lambda\}\\
&\leq\p\{\sigma^+\vee\sigma^-<T\}\leq\frac{\E M(T)}{\lambda^2}\leq\frac{\E M(0)}{\lambda^2}\\
&=\frac{1}{\lambda^2}(y_n(u+h,0)-y_n(u,0))(y_n(u,0)-y_n(u-h,0))\leq\frac{9h^2}{\lambda^2}.
\end{align*}
\end{proof}

\begin{lem}\label{lemma_first_points}
For all $\beta>1$
$$
\lim_{\delta\to 0}\sup\limits_{n\geq 1}\E\left[\|y_n(\delta,\cdot)-y_n(0,\cdot)\|^{\beta}\wedge 1\right]=0.
$$
\end{lem}

\begin{proof}
Set
$$
\sigma_{\delta}=\{t:\ y_n(\delta,t)-y_n(0,t)=1\}\wedge T.
$$
The assertion of the lemma follows from the inequalities
\begin{align*}
\E\Bigg[\sup_{t\in[0,T]}(y_n(\delta,t)&-y_n(0,t))^{\beta}\wedge 1\Bigg] =\E\sup_{t\in[0,T]}(y_n(\delta,t\wedge\sigma_{\delta})-y_n(0,t\wedge\sigma_{\delta}))^{\beta}\\
&\leq C_{\beta}\E(y_n(\delta,T\wedge\sigma_{\delta})-y_n(0,T\wedge\sigma_{\delta}))^{\beta}\\
&\leq C_{\beta}\E(y_n(\delta,T\wedge\sigma_{\delta})-y_n(0,T\wedge\sigma_{\delta}))\leq C_{\beta}\delta.
\end{align*}
\end{proof}

\begin{lem}\label{lemma_enequality_for_stop_time}
Let $\xi(t),\ t\in[0,T]$, be a continuous local square integrable martingale starting from 0 and $w(t),\ t\geq 0$, be a Wiener process. Denote for a fixed $a\in\R$
$$
\tau=\inf\{t:\ \xi(t)=a\}\wedge T.
$$
If there exists a constant $b>0$ such that
$$
\langle\xi(\cdot)\rangle_t\geq bt,\ t\in[0,\tau],
$$
then
$$
\p\{\tau\geq t\}\leq\p\{\sigma\geq t\},\quad t\in[0,T],
$$
and
$$
\E\tau\leq\E\sigma,
$$
where $\sigma=\inf\{t:\ w(bt)=a\}\wedge T$.
\end{lem}

\begin{proof}
{\col Since a continuous local martingale is necessarily a local square integrable martingale, there exists a Wiener process $\widetilde{w}(t),\ t\geq 0$, such that
$$
\xi(t)=\widetilde{w}(\langle\xi(\cdot)\rangle_t),
$$
by Theorem~2.7.2'~\cite{Watanabe:1981:en}.}
Denote
$$
\widetilde{\sigma}=\inf\{t:\ \widetilde{w}(bt)=a\}.
$$
It is easy to see that $b\widetilde{\sigma}\geq\langle\xi(\cdot)\rangle_{\tau}\geq b\tau$. Since $\tau\leq T$, $\tau\leq\widetilde{\sigma}\wedge T$. The latter inequality proves the lemma.
\end{proof}

\begin{lem}\label{lemma_tightness_in_C}
For all $\eps>0$ and $u\in[0,1]$ the sequence $\{y_n(u,t),\ t\in[\eps,T]\}_{n\geq 1}$ is tight in $C[\eps,T]$.
\end{lem}

\begin{proof}
To prove the lemma we use the Aldous tightness criterion (see e.g. Theorem~3.6.4.~\cite{Dawson:1993}). So, for the tightness of $\{y_n(u,t),\ t\in[\eps,T]\}_{n\geq 1}$ in the space $C[\eps,T]$ we have to check the following properties
\begin{enumerate}
\item[(A1)] for all $t\in[\eps,T]$ the sequence $\{y_n(u,t)\}_{n\geq 1}$ is tight in $\R$;

\item[(A2)] for all $r>0$, a set of stopping times $\{\sigma_n\}_{n\geq 1}$ taking values in $[\eps,T]$ and a sequence $\delta_n\searrow 0$
$$
\lim_{n\to\infty}\p\{|y_n(u,\sigma_n+\delta_n)-y_n(u,\sigma_n)|\geq r\}=0.
$$
\end{enumerate}

Note that property $(A1)$ follows from Lemma~\ref{lemma_expectation} and Chebyshev's inequality. In order to prove $(A2)$, we will first estimate the probability of the event $\{m_n(u,\eps)<\gamma\}$, for all $\gamma\in\left(0,\frac{1}{2}\right)$. We can assume, without loss of generality, that $u\in\left[0,\frac{1}{2}\right]$. {\col Set $\xi(t)=y_n(u+\gamma,t)-y_n(u,t)$, $t\in[0,T]$, and note that
$$
\langle\xi(\cdot)\rangle_t=\int_0^t\left(\frac{1}{m_n(u+\gamma,s)}+\frac{1}{m_n(u,s)}\right)ds\geq 2t,\ \ t\in[0,\tau^n_{u,u+\gamma}].
$$
So, using Lemma~\ref{lemma_enequality_for_stop_time} with $b=2$, we obtain
$$
\p\{m_n(u,\eps)<\gamma\}\leq\p\{\tau^n_{u,u+\gamma}\geq\eps\}\leq\p\{\widetilde{\tau}_{u,u+\gamma}\geq\eps\}\leq C\gamma,
$$
}
where
\begin{equation}\label{f_m_n}
m_n(u,\cdot)=\begin{cases}
m_k^n(\cdot),\quad \frac{k-1}{n}\leq u<\frac{k}{n},\ \ k\in[n],\\
m_n^n(\cdot),\quad u=1,
\end{cases}
\end{equation}
\begin{equation}\label{f_time_of_meeting_n}
\tau^n_{u,v}=\inf\{t:\ y_n(u,t)=y_n(v,t)\}\wedge T,\quad u,v\in[0,1],
\end{equation}
and $\widetilde{\tau}_{u,v}$ is a time of meeting of two independent Wiener processes starting from $u$ and $v$ respectively. Let $\eps_1>0$ be fixed. Choose $\gamma>0$ so that
$$
\p\{m_n(u,\eps)<\gamma\}<\frac{\eps_1}{2}.
$$
Next, set
$$
\xi_n(t)=y_n(u,\sigma_n+t)-y_n(u,\sigma_n)
$$
and estimate the following probability
$$
\p\{|\xi_n(t)|\geq r\}\leq\p\left\{|\xi_n(t)|\I_{\{m_n(u,\eps)\geq\gamma\}}\geq r\right\}+\p\{m_n(u,\eps)<\gamma\}.
$$
To prove the smallness of the first term for large $n$ we will show that $\xi_n(\cdot)\I_{\{m_n(u,\eps)\geq\gamma\}}$ is an $(\F_{\sigma_n+t}^{y_n})$-martingale. So, assume $s\leq t$ and consider
\begin{align*}
\E\big(\xi_n(&t)\I_{\{m_n(u,\eps)\geq\gamma\}}\big|\F_{\sigma_n+s}^{y_n}\big)\\ &=\E\left(\left.(y_n(u,\sigma_n+t)-y_n(u,\sigma_n))\I_{\{m_n(u,\eps)\geq\gamma\}}\right|\F_{\sigma_n+s}^{y_n}\right)\\
&=\I_{\{m_n(u,\eps)\geq\gamma\}}\E\left(\left.y_n(u,\sigma_n+t)-y_n(u,\sigma_n)\right|\F_{\sigma_n+s}^{y_n}\right)\\
&=\I_{\{m_n(u,\eps)\geq\gamma\}}(y_n(u,\sigma_n+s)-y_n(u,\sigma_n))=\xi_n(s)\I_{\{m_n(u,\eps)\geq\gamma\}}.
\end{align*}
Here we used the optional sampling theorem (see e.g. Theorem~1.6.11~\cite{Watanabe:1981:en}) and the inclusion $\F_{\eps}^{y_n}\subseteq\F_{\sigma_n+s}^{y_n}$. Now we are ready to estimate
\begin{align*}
\p\bigg\{|\xi_n(&\delta_n)|\I_{\{m_n(u,\eps)\geq\gamma\}}\geq r\bigg\}\leq\frac{1}{r^2}\E\left[\xi_n^2(\delta_n)\I_{\{m_n(u,\eps)\geq\gamma\}}\right]\\
&=\frac{1}{r^2}\E\left[\E\left(\left.(y_n(u,\sigma_n+\delta_n)-y_n(u,\sigma_n))^2\I_{\{m_n(u,\eps)\geq\gamma\}}\right| \F_{\sigma_n}^{y_n}\right)\right]\\
&=\frac{1}{r^2}\E\left[\I_{\{m_n(u,\eps)\geq\gamma\}}\E\left(\left.(y_n(u,\sigma_n+\delta_n)-y_n(u,\sigma_n))^2\right| \F_{\sigma_n}^{y_n}\right)\right]
\end{align*}
\begin{align*}
&=\frac{1}{r^2}\E\left[\I_{\{m_n(u,\eps)\geq\gamma\}}\E\left(\left.\int_{\sigma_n}^{\sigma_n+\delta_n}\frac{ds}{m_n(u,s)}\right| \F_{\sigma_n}^{y_n}\right)\right]\\
&=\frac{1}{r^2}\E\left[\I_{\{m_n(u,\eps)\geq\gamma\}}\int_{\sigma_n}^{\sigma_n+\delta_n}\frac{ds}{m_n(u,s)}\right]\\
&\leq\frac{4}{r^2}\E\left[\int_{\sigma_n}^{\sigma_n+\delta_n}\frac{ds}{\gamma\vee m_n(u,s)}\right]\leq\frac{\delta_n}{r^2\gamma}.
\end{align*}
Thus, there exists $N\in\N$ such that for all $n\geq N$ $\p\left\{|\xi_n(\delta_n)|\I_{\{m_n(u,\eps)\geq\gamma\}}\geq r\right\}<\frac{\eps_1}{2}$. This completes the proof.
\end{proof}

\begin{proof}[Proof of Proposition~\ref{prop_tightness_in_D_eps}]
By Lemmas~\ref{lemma_three_points},~\ref{lemma_first_points} and \ref{lemma_tightness_in_C}, Theorems~3.8.6 and~3.8.8~\cite{Ethier:1986}, Remark~3.8.9~\cite{Ethier:1986}, we obtain the assertion of the proposition.
\end{proof}

Proposition~\ref{prop_tightness_in_D_eps} and Proposition~\ref{prop_tightness_apend} immediately imply the tightness of $\{y_n(u,t),\ u\in[0,1],\ t\in(0,T]\}$.

\begin{prp}
The sequence $\{y_n(u,t),\ u\in[0,1],\ t\in(0,T]\}$ is tight in $D([0,1],C(0,T])$.
\end{prp}

\subsection{Extension to the space $D([0,1],C[0,T])$}\label{section_extension}

It should be noted that since the sequence $\{y_n(u,t),\ u\in[0,1],\ t\in(0,T]\}$ is tight in the separable metric space $D([0,1],C(0,T])$, it has limit points (in the weak topology), by Prokhorov's theorem. In this section we will show that every limit point of the sequence can be extended to the space $D([0,1],C[0,T])$. Denote by $C_b^2(\R)$ the set of twice continuously differentiable functions on $\R$ which are bounded together with their derivatives.

\begin{prp}\label{prop_tight_xi}
Let $\varphi\in C_b^2(\R)$ and
\begin{equation}\label{f_xi}
\xi_n(t)=\int_0^1\varphi(y_n(u,t))du,\quad t\in[0,T].
\end{equation}
Then the sequence $\{\xi_n(t),\ t\in[0,T]\}_{n\geq 1}$ is tight in $C[0,T]$.
\end{prp}

\begin{proof}
We use the Aldous tightness criterion to prove the proposition. By the boundedness of $\varphi$, the sequence $\{\xi_n(t)\}_{n\geq 1}$ is bounded for each $t\in[0,T]$. Hence, it is enough to check that for all $\eps>0$, a set of stopping times $\{\sigma_n\}_{n\geq 1}$ on $[0,T]$ and a sequence $\delta_n\searrow 0$ one has
\begin{equation}\label{f_different}
\lim_{n\to\infty}\p\{|\xi_n(\sigma_n+\delta_n)-\xi_n(\sigma_n)|\geq\eps\}=0.
\end{equation}
To show~\eqref{f_different} we consider the difference $\xi_n(\sigma_n+t)-\xi_n(\sigma_n)$ and use Ito's formula. So, we obtain
\begin{align*}
\xi_n(\sigma_n&+t)-\xi_n(\sigma_n)=\frac{1}{n}\sum_{k=1}^n[\varphi(x_k^n(\sigma_n+t))-\varphi(x_k^n(\sigma_n))]\\
&=\frac{1}{n}\sum_{k=1}^n\int_0^t\dot{\varphi}(x_k^n(\sigma_n+s))dx_k^n(\sigma_n+s)
\end{align*}
$$
+\frac{1}{2n}\sum_{k=1}^n\int_0^t\frac{\ddot{\varphi}(x_k^n(\sigma_n+s))}{m_k^n(\sigma_n+s)}ds=M(t)+\frac{1}{2}A(t).
$$
Next, estimate $\E|A(t)|$ and $\E|M(t)|$.

Denote the number of distinct points $x_k^n(\sigma_n+t)$, $k\in[n]$, by $\chi_n(t)$, i.e.
$$
\chi_n(t)=|\{x_k^n(\sigma_n+t),\ k\in[n]\}|.
$$
So,
\begin{gather*}
\E|A(t)|\leq\frac{\|\ddot{\varphi}\|}{n}\sum_{k=1}^n\int_0^t\frac{ds}{m_k^n(\sigma_n+s)}= \|\ddot{\varphi}\|\E\int_0^t\chi_n(s)ds=\|\ddot{\varphi}\|\sum_{k=1}^n\E\gamma_k^n(t),
\end{gather*}
where
\begin{align*}
\gamma_1^n(t)&=t,\\
\gamma_k^n(t)&=\inf\{s:\ x_k^n(\sigma_n+s)=x_{k-1}^n(\sigma_n+s)\}\wedge t,\ \ k=2,\ldots,n.
\end{align*}

Let $\{z_k^n(t),\ t\in[0,T],\ k\in[n]\}$ be the set of coalescing Brownian particles starting from non-random points and possessing $z_k^n(0)\leq z_{k+1}^n(0)$, $k\in[n-1],\ t\in[0,T]$, and let $\{\widetilde{x}_k^n(t),\ t\in[0,T],\ k\in[n]\}$ be the set of processes which satisfy conditions $(F1)$, $(F3)$--$(F5)$. Define $\widehat{\gamma}^n_k(t),\ k\in[n]$, and $\widetilde{\gamma}^n_k(t),\ k\in[n]$, in the same way as $\gamma_k^n(t),\ k\in[n]$, replacing $x_k^n(\sigma_n+\cdot)$ with $z_k^n(\cdot)$, $k\in[n]$, and $x_k^n(\sigma_n+\cdot)$ with $\widetilde{x}_k^n(\cdot)$, $k\in[n]$, respectively.

Using Markov's property of $\{x_k^n(t),\ t\in[0,1],\ k\in[n]\}$~\cite{Konarovskiy:2011:TSP} and Lemma~\ref{lemma_enequality_for_stop_time} we have
$$
\E\gamma_k^n(t)=\E(\E(\gamma_k^n(t)|\F_{\sigma_n}^n))= \E(\E_{x^n(\sigma_n)}\widetilde{\gamma}_k^n(t))\leq\E(\E_{x^n(\sigma_n)}\widehat{\gamma}_k^n(t)).
$$
It is well known that there exists a constant $C$, which does not depend on $n$ and $z_k^n(0),\ k\in[n]$, such that for all $t\in[0,T]$
$$
\E\left(\sum_{k=1}^n\widehat{\gamma}_k^n(t)\right)\leq C(z_n^n(0)-z_1^n(0))\sqrt{t}.
$$
(see Section~7.1.~\cite{Dorogovtsev:2007:en}).
Thus
$$
\E|A(t)|\leq\|\ddot{\varphi}\|\sum_{k=1}^n\E(\E_{x^n(\sigma_n)}\widehat{\gamma}_k^n(t))\leq C\E(x_n^n(\sigma_n)-x_1^n(\sigma_n))\sqrt{t} \leq C\sqrt{t}.
$$

Next, consider
\begin{align*}
(\E|M(&t)|)^2\leq\E M^2(t)=\E\left(\frac{1}{n}\sum_{k=1}^n\int_0^t\dot{\varphi}(x_k^n(\sigma_n+s))dx_k^n(\sigma_n+s)\right)^2\\
&=\frac{1}{n^2}\sum_{k=1}^n\sum_{l=1}^n \E\int_0^t \frac{\dot{\varphi}(x_k^n(\sigma_n+s))\dot{\varphi}(x_l^n(\sigma_n+s))}{m_k^n(\sigma_n+s)}\I_{\{\tau^n_{k,l}\leq\sigma_n+s\}}ds\\
&=\frac{1}{n}\sum_{k=1}^n\E\int_0^t \dot{\varphi}^2(x_k^n(\sigma_n+s))\sum_{l=1}^n \frac{\I_{\{\tau^n_{k,l}\leq\sigma_n+s\}}}{nm_k^n(\sigma_n+s)}ds\\
&=\frac{1}{n}\sum_{k=1}^n\E\int_0^t\dot{\varphi}^2(x_k^n(\sigma_n+s))ds\leq \|\dot{\varphi}^2\|t.
\end{align*}

Now from the obtained estimations of  $\E|A(t)|$ and $\E|M(t)|$ and Chebyshev's inequality we have~\eqref{f_different}. The proposition is proved.
\end{proof}

Note that, since the space $D([0,1],C(0,T])$ is separable, there exists a sequence $\{n'\}$ and a random element $\{y(u,t),\ u\in[0,1],\ t\in(0,T]\}$ in this space such that $y_{n'}$ tends to $y$ in distribution in $D([0,1],C(0,T])$, by Prokhorov's theorem~\cite{Billingsley:1999}. Next, from Skorohod's theorem (see Theorem~3.1.8~\cite{Ethier:1986}) we have the following result.

\begin{lem}\label{lemma_converg_a_s}
There exists a probability space with random elements $\{\widetilde{y}_{n'}(u,t),\\ u\in[0,1],\ t\in(0,T]\}_{n'}$ and $\{\widetilde{y}(u,t),\ u\in[0,1],\ t\in(0,T]\}$ taking values in $D([0,1],C(0,T])$ such that
\begin{enumerate}
\item[1)] for all $n'$ $\mathrm{Law}(\widetilde{y}_{n'})=\mathrm{Law}(y_{n'})$

\item[2)] $\widetilde{y}_{n'}\to\widetilde{y}$ in $D([0,1],C(0,T])$ a.s.
\end{enumerate}
\end{lem}

\begin{rmk}
Since for every $n'$, $\{\widetilde{y}_{n'}(u,t),\ u\in[0,1],\ t\in(0,T]\}$ satisfies $1)$ of Lemma~\ref{lemma_converg_a_s}, we may suppose that $\widetilde{y}_{n'}$ is a random element in $D([0,1],C[0,T])$ considering $\widetilde{y}_{n'}(\cdot,0)=y_{n'}(\cdot,0)$.
\end{rmk}

\begin{rmk}
By Lemma~\ref{lemma_convergence_in_D}, property $2)$ is equivalent to
\begin{enumerate}
\item[2')] for all $\eps>0$, $\widetilde{y}_{n'}(\cdot,\eps\vee\cdot)\to\widetilde{y}(\cdot,\eps\vee\cdot)$ in $D([0,1],C[0,T])$ a.s.
\end{enumerate}
\end{rmk}

\begin{rmk}
Note that the index $n$ in $y_n$ or in $x_{\cdot}^n$ means that we have a system of particles which start from the set of the points $\frac{k}{n},\ k\in[n]$, with the mass $\frac{1}{n}$. Since we are not going to use this fact any more, for convenience we will suppose that $\widetilde{y}_{n'}=y_n$ and $\widetilde{y}=y$, namely we will suppose that for each $\eps>0$
$$
y_n(\cdot,\eps\vee\cdot)\to y(\cdot,\eps\vee\cdot)\ \ \mbox{in}\ D([0,1],C[0,T])\ \ \mbox{a.s.}
$$
\end{rmk}

We are going to extend $\{y(u,t),\ u\in[0,1],\ t\in(0,T]\}$ to $D([0,1],C[0,T])$. We will be able to do it due to the following lemmas.

\begin{lem}\label{lemma_prop_y_in_0}
Let $\varphi\in C_b^2(\R)$ and $y(u,0)=u$, $u\in[0,1]$. Then the random process $\xi(t)=\int_0^1\varphi(y(u,t))du,\ t\in[0,T]$,
is continuous a.s.
\end{lem}

\begin{proof}
Let us consider the map $F_{\varphi}:D([0,1],C(0,T])\to C(0,T]$
$$
F_{\varphi}(g)(t)=\int_0^1\varphi(g(u,t))du,\quad t\in(0,T],\ \ g\in D([0,1],C(0,T]).
$$
It is easy to see that $F_{\varphi}$ is continuous. Let $\xi_n(t)$, $t\in[0,T]$, be defined by~\eqref{f_xi}. Thus,
$$
\xi_n=F_{\varphi}(y_n)\to\xi\ \ \mbox{in}\ C(0,T]\ \ \mbox{a.s.}
$$
On the other hand, by Proposition~\ref{prop_tight_xi}, $\{\xi_n\}_{n\geq 1}$ is tight in $C[0,T]$. It implies that $\xi$ is continuous on $[0,T]$ a.s.
\end{proof}

{\col Let $C_K(\R)$ denote the class of continuous functions on $\R$ with compact support and $\h$ be a dense subset of $C_K(\R)$.

\begin{lem}\label{lemma_expectation_of_z}
  Let $\{z(u,t),\ u\in[0,1],\ t\in(0,T]\}\in D([0,1],C(0,T])$, $z(u,t)\leq z(v,t)$, $u<v$, $t\in(0,T]$, and for all $\varphi\in\h$
  \begin{equation}\label{f_converg_of_int}
  \int_0^1\varphi(z(u,t))du\to\int_0^1\varphi(u)du,\quad t\to 0.
  \end{equation}
  Then for each $u\in(0,1)$, $\lim_{t\to 0}z(u,t)=u$. Moreover, if $\lim_{t\to 0}z(0,t)=0$ and $\lim_{t\to 0}z(1,t)=1$, then the extension of $z$ on $[0,T]$
  $$
  z(u,t)=\begin{cases}
           z(u,t),& t\in(0,T],\\
           u& t=0,
         \end{cases}\quad u\in[0,1],
  $$
  belongs to $D([0,1],C[0,T])$.
\end{lem}

\begin{proof}
  Note that the density of $\h$ implies that~\eqref{f_converg_of_int} holds for all $\varphi\in C_K(\R)$. Let $\{t_n\}_{n\geq 1}\subset(0,T]$ be a sequence that converges to zero. Next we will consider $z(\cdot,t_n)$, $n\in\N$, as random elements in the probability space $([0,1],\B(\R),\Leb)$. For each $n\in\N$ we denote the distribution of $z(\cdot,t_n)$ in $\R$ by $\mu_n$, i.e. $\mu_n=\Leb\circ z(\cdot,t_n)^{-1}$. Then $\mu_n$ converges vaguely to $\mu$, where $\mu$ coincides with the Lebesgue measure on $[0,1]$ and $\mu([0,1]^c)=0$. Since $\Leb(\R)=1$, $\mu_n$ converges to $\Leb$ weakly, by Lemma~5.10~\cite{Kallenberg:2002}. Set
  $$
  F_n(x)=\Leb\{u:\ z(u,t_n)\leq x\},\quad x\in\R.
  $$
  Then using Theorem~2.1~\cite{Billingsley:1999}, we get
  \begin{equation}\label{f_conv_F}
  F_n(x)=\mu_n((-\infty,x])\to \mu((-\infty,x])=x\ \ \mbox{as}\ n\to\infty,
  \end{equation}
  for all $x\in[0,1]$.
  By the monotonicity and the right continuity of $z(\cdot,t_n)$, for all fixed $n$,
  $$
  z(u,t_n)=\sup\{x:\ F_n(x)\leq u\},\quad u\in(0,1).
  $$
  Next, let $u\in(0,1)$ and $\eps>0$ be fixed and $u+\eps,u-\eps\in(0,1)$. By~\eqref{f_conv_F}, there exists $N$ such that for all $n\geq N$,
  $F_n(u-\eps)\leq u< F_n(u+\eps)$. Hence
  $$
  u-\eps\leq z(u,t_n)\leq u+\eps,\quad \mbox{if }\ n\geq N.
  $$
  Since $\eps>0$ was arbitrary small, we have that $z(\cdot,t_n)\to u$, for all $u\in(0,1)$. Hence, for all $u\in(0,1)$,
  $$
  \lim_{t\to 0}z(u,t)=u,\quad u\in(0,1).
  $$
  This proves the first part of the lemma.

  Let $\lim_{t\to 0}z(0,t)=0$ and $\lim_{t\to 0}z(1,t)=1$. Then for each sequence $\{v_n\}_{n\geq 1}$ decreasing to $u$ and for all $t\in[0,T]$, $z(v_n,t)\downarrow z(u,t)$. By Dini's theorem $z(v_n,\cdot)\to z(u,\cdot)$ in $C[0,T]$. So, $z(u,\cdot)$, $u\in[0,1]$, is right continuous. Using Dini's theorem for an increasing sequence $\{v_n\}_{n\geq 1}$ again, we obtain that $z(u,\cdot)$, $u\in[0,1]$, has left limits.
\end{proof}
}
Let us prove the main result of this section.

\begin{prp}\label{prop_prolong}
Set
$$
y(u,0)=u,\ \ u\in[0,1].
$$
Then $\{y(u,t),\ u\in[0,1],\ t\in[0,T]\}$ is a random element in $D([0,1],C[0,T])$. Furthermore, for all $u\in[0,1]$ the process $y(u,\cdot)$ is a continuous $(\F_t)$-martingale, where
$$
\F_t=\sigma(y(u,s),\ u\in[0,1],\ s\leq t).
$$
\end{prp}

To prove the proposition, we first show that $y(u,\cdot\vee\eps)$ is a martingale, for all $u\in[0,1]$. Then using Lemmas~\ref{lemma_prop_y_in_0} and~\ref{lemma_expectation_of_z} we extend $y$ to $D([0,1],C[0,T])$. To prove that $y(u,\cdot\vee\eps)$ is a martingale we are going to use the fact that $\{y_n(u,\cdot\vee\eps)\}_{n\geq 1}$ is also a martingale. It should be noted that in general, property $2')$ does not imply the convergence of $\{y_n(u,\cdot\vee\eps)\}_{n\geq 1}$ to $y(u,\cdot\vee\eps)$. So, we need {\col the following result, which by a standard property of the Skorohod topology then does give this convergence (see Corollary~\ref{coroll_mart_prop} below).}

\begin{lem}\label{lemma_cont}
For all $\eps>0$ and $u\in[0,1]$ one has
$$
\p\{y(u,\eps\vee\cdot)\neq y(u-,\eps\vee\cdot)\}=0.
$$
\end{lem}

\begin{proof}
Let $u\in[0,1]$ and $\eps>0$ be fixed. Since for all $n\geq 1$, $y_n(\cdot,\eps\vee\cdot)$ is non-decreasing in the first argument, $y(\cdot,\eps\vee\cdot)$ is non-decreasing too. So, for $\gamma>0$, $\delta>0$ and $\beta>1$ we have
\begin{align*}
\p\{\|y(u,\eps&\vee\cdot)-y(u-\delta,\eps\vee\cdot)\|>\gamma\}\\
&\leq \p\{\|y(u+\delta,\eps\vee\cdot)-y(u-\delta,\eps\vee\cdot)\|>\gamma\}\\ &\leq\p\left\{\bigcup_{n=1}^{\infty}\bigcap_{k=n}^{\infty}\{\|y_k(u+2\delta,\eps\vee\cdot)-y_k(u-2\delta,\eps\vee\cdot)\|>\gamma\}\right\}\\
&\leq\lim_{n\to\infty} \p\left\{\bigcap_{k=n}^{\infty}\{\|y_k(u+2\delta,\eps\vee\cdot)-y_k(u-2\delta,\eps\vee\cdot)\|>\gamma\}\right\}\\
&\leq\varlimsup_{n\to\infty}\p\{\|y_n(u+2\delta,\eps\vee\cdot)-y_n(u-2\delta,\eps\vee\cdot)\|>\gamma\}\\
&\leq\frac{1}{\gamma^{\beta}}\varlimsup_{n\to\infty}\E\left[\|y_n(u+2\delta,\cdot)-y_n(u-2\delta,\cdot)\|^{\beta}\wedge 1\right]\leq\frac{4\delta C_{\beta}}{\gamma^{\beta}}.
\end{align*}

Next, passing to the limit as $\delta$ tends to $0$ and using the monotonicity of $\{\|y(u,\eps\vee\cdot)-y(u-\delta,\eps\vee\cdot)\|>\gamma\}$ in $\delta$ we obtain
$$
\p\{\|y(u,\eps\vee\cdot)-y(u-,\eps\vee\cdot)\|>\gamma\}=0.
$$
This proves the lemma.
\end{proof}

\begin{cor}\label{coroll_mart_prop}
For all $\eps>0$ and $u\in[0,1]$
$$
y_n(u,\eps\vee\cdot)\to y(u,\eps\vee\cdot)\ \ \mbox{in}\ \ C[0,T]\ \ \mbox{a.s.}
$$
\end{cor}

\begin{proof}[Proof of Proposition~\ref{prop_prolong}]
Set
$$
\F_t^{\eps}=\sigma(y(u,s),\ u\in[0,1],\ s\in[\eps,t\vee\eps]),\quad t\in[0,T].
$$
Reasoning as in the proof of Proposition~9.1.17~\cite{Jacod:2003}, we can prove, using Corollary~\ref{coroll_mart_prop}, that $y(u,t\vee\eps),\ t\in[0,T]$, is an $(\F_t^{\eps})$-local martingale, for all $u\in[0,1]$. {\col By Lemma~\ref{lemma_expectation}~$(i)$ and Fatou's lemma, there exists a constant $C$, independent on $u,\ t$ and $\eps$, such that
\begin{equation}\label{f_estim_exp_y}
\E|y(u,t)|\leq C,\quad u\in[0,1],\ \ t\in[0,T].
\end{equation}
Using the second part of the Lemma~\ref{lemma_expectation}, it is easy to see that $y(u,t\vee\eps),\ t\in[0,T]$, is an $(\F_t^{\eps})$-martingale for all $u\in(0,1)$. Since for each $t$ and $\eps$, $y(\cdot,t\vee\eps)$ is non-decreasing with respect to the first variable and continuous at $u=0$ and $u=1$, $y(u,t\vee\eps),\ t\in[0,T]$, is an $(\F_t^{\eps})$-martingale for $u=0$ and $u=1$, by the monotone convergence theorem and~\eqref{f_estim_exp_y}.}

Note that
$$
\F_t=\sigma\left(\bigcup_{0<\eps\leq t}\F_t^{\eps}\right),\quad t\in(0,T].
$$
Take $\eps\leq s\leq t$ and consider
$$
\E\left(\left. y(u,t\vee\eps)\right|\F_s^{\eps}\right)=y(u,s\vee\eps),
$$
that is equivalent to
$$
\E\left(\left. y(u,t)\right|\F_s^{\eps}\right)=y(u,s).
$$
{\col By Levy's theorem} (see e.g. Theorem~1.5~\cite{Liptser:2001})
$$
\E\left(\left. y(u,t)\right|\F_s^{\eps}\right)\to\E\left(\left. y(u,t)\right|\F_s\right)\quad\mbox{a.s.},\ \ \eps\to 0
$$
so
\begin{equation}\label{f_mart_prop}
\E\left(\left. y(u,t)\right|\F_s\right)=y(u,s)
\end{equation}
and it means that $y(u,t),\ t\in(0,T]$, is an $(\F_t)$-martingale.

Let $\h\subset C_b^2(\R)$ be a {\col countable dense subset of $C_K(\R)$}. By Lemma~\ref{lemma_prop_y_in_0}, there exists $\Omega'$ such that $\p\{\Omega'\}=1$ and for all $\omega\in\Omega'$ and {\col $\varphi\in\h$, $\int_0^1\varphi(y(u,t,\omega))du$, $t\in[0,T]$, is continuous.} Consequently, $\int_0^1\varphi(y(u,t,\omega))du$ converges to $\int_0^1\varphi(u)du$ as $t\to 0$. Hence Lemma~\ref{lemma_expectation_of_z} implies the continuity of $y(u,t,\omega)$, $t\in[0,T]$, for all $u\in(0,1)$ and $\omega\in\Omega'$.

Next, using Lemma~\ref{lemma_cont} and the martingale property of $y(u,\cdot)$, we will show that $\lim_{t\to 0}y(1,t)=1$ and $\lim_{t\to 0}y(0,t)=0$ a.s. Thus, we will be able to use the second part of Lemma~\ref{lemma_expectation_of_z} in order to state that $\{y(u,t),\ u\in[0,1],\ t\in[0,T]\}\in D([0,1],C[0,T])$ a.s.

Since $y(u,t),\ t\in(0,T]$, is an $(\F_t)$-martingale
$$
y(u,t)=\E\left(\left. y(u,T)\right|\F_t\right),\quad u\in[0,1].
$$
{\col By Levy's theorem},
$$
\E\left(\left. y(u,T)\right|\F_t\right)\to\E\left(\left. y(u,T)\right|\F_{0+}\right)\quad\mbox{a.s.},\ \ t\to 0.
$$
So, $\E\left(\left. y(u,T)\right|\F_{0+}\right)=u$ a.s. for all $u\in(0,1)$. Let us prove that $\E\left(\left. y(0,T)\right|\F_{0+}\right)=0$ and $\E\left(\left. y(1,T)\right|\F_{0+}\right)=1$.

Let $v_n\downarrow 0$ then by Lemma~\ref{lemma_cont} and the monotone convergence theorem for conditional expectations (see e.g. Theorem~1.1~\cite{Liptser:2001})
$$
0=\lim_{n\to\infty}v_n=\lim_{n\to\infty}\E(y(v_n,T)|\F_{0+})=\E(y(0,T)|\F_{0+}),
$$
where we understand $\lim$ as the limit almost surely. Similarly, $$\E\left(\left. y(1,T)\right|\F_{0+}\right)=1.$$ Thus by Lemma~\ref{lemma_expectation_of_z}, $\{y(u,t),\ u\in[0,1],\ t\in[0,T]\}$ belongs to $D([0,1],C[0,T])$ a.s. Moreover, since for each $u\in[0,1]$, $y(u,\cdot)$ is a random element of $C[0,T]$, $\{y(u,t),\ u\in[0,1],\ t\in[0,T]\}$ is a random element of $D([0,1],C[0,T])$, by Proposition~3.7.1~\cite{Ethier:1986}.

Note that the martingale property of $y(u,t),\ t\in[0,T]$, follows from~\eqref{f_mart_prop} and the equality $\E\left(\left. y(u,T)\right|\F_{0+}\right)=u$, $u\in[0,1]$. The proposition is proved.
\end{proof}

\subsection{Verification of conditions $(C4)$ and $(C5)$}\label{section_C4_C5}

In this section we will show that the process $\{y(u,t),\ u\in[0,1],\ t\in[0,T]\}$ satisfies conditions $(C4)$ and $(C5)$ of Theorem~\ref{theorem_main_result}. To check this, first we state some properties of a sequence of stopping times which are defined by a sequence of local martingales, then we will prove that the joint local quadratic variation of $y(u,\cdot)$ and $y(v,\cdot)$ is the limit of the joint local quadratic variation of $y_n(u,\cdot)$ and $y_n(v,\cdot)$.

It should be noted that $\{y(u,t),\ u\in[0,1],\ t\in[0,T]\}$ satisfies conditions $(C2)$ and $(C3)$ (see Proposition~\ref{prop_prolong}). Let us prove an auxiliary lemma.

\begin{lem}\label{lemm_conv}
Let $z_n(t),\ t\in[0,T],\ n\geq 1$, be {\col a set of continuous local martingales such that} {\cob for all $n\geq 1$ and $s,t\in[0,\tau_n]$, $s<t$,
\begin{equation}\label{rest_char}
\langle z_n(\cdot)\rangle_t-\langle z_n(\cdot)\rangle_s\geq p(t-s),
\end{equation}
}
where $\tau_n=\inf\{t:\ z_n(t)=0\}\wedge T$ and $p$ is a non-random positive constant. Let $z(t),\ t\in[0,T],$ be a continuous process such that
$$
z(\cdot\wedge\tau)=\lim_{n\to\infty}z_n(\cdot\wedge\tau_n)\ (\mbox{in}\ \ C([0,T],\R)) \ a.s.,
$$
where $\tau=\inf\{t:\ z(t)=0\}\wedge T$. Then
\begin{equation}\label{len_cen}
\tau=\lim_{n\to\infty}\tau_n\ \mbox{in probability}.
\end{equation}
\end{lem}

{\cob
\begin{rmk}
If the function $\langle z_n(\cdot)\rangle_t$, $t\in[0,T]$, is absolutely continuous with $\frac{d}{dt}\langle z_n(\cdot)\rangle_t\geq p$, $t\in(0,T]$, then it satisfies~\eqref{rest_char}.
\end{rmk}
}

\begin{proof}
Set $A=\{\omega:\ z(\cdot\wedge\tau(\omega),\omega)=\lim_{n\to\infty}z_n(\cdot\wedge\tau_n(\omega),\omega)\}$. Take $\eps>0$, $\omega\in A$ and suppose that $\tau(\omega)>0$. Denote the subset $\{z(t,\omega):\ t\in[0,(\tau(\omega)-\eps)\vee 0]\}$ of $\R$ by $K_{\eps}(\omega)$.

Since $z$ is a continuous process, $K_{\eps}(\omega)$ is a compact set as image of a compact set and $0\not\in K_{\eps}(\omega)$. Hence there exists $\delta(\omega)>0$ such that $0\not\in K_{\eps}^{\delta}(\omega)$, where
$$
K_{\eps}^{\delta}(\omega)=\{a\in\R:\ \inf\{|a-b|,\ b\in K_{\eps}(\omega)\}<\delta\}.
$$
Since $z(\cdot\wedge\tau(\omega),\omega)=\lim_{n\to\infty}z_n(\cdot\wedge\tau_n(\omega),\omega)$, there exists $N(\omega)$ such that for every $n\geq N(\omega)$ $z_n(t,\omega)\in K_{\eps}^{\delta}(\omega)$, $t\in[0,(\tau(\omega)-\eps)\vee 0]$. It implies the inequality
$$
\tau_n(\omega)\geq\tau(\omega)-\eps.
$$
For the case $\tau(\omega)=0$ the latter inequality is obvious. So, we have
$$
\lim_{n\to\infty}\p\{\tau-\tau_n\geq\eps\}=0.
$$

To prove the equality $\lim_{n\to\infty}\p\{\tau_n-\tau\geq\eps\}=0$ we need the following lemma.

\begin{lem}\label{lemm_sup}
Let $w(t),\ t\geq 0$, be a Wiener process and $\sigma_x=\inf\{t:\ w(t)=x\}$. Then for every $\eps>0$
$$
\sup\limits_{x\in\R}\p\left\{\sigma_x\geq\eps,\ \sup\limits_{t\in[\sigma_x-\eps,\sigma_x]}|w(t)-x|<\delta\right\}\to 0\ \ \mbox{as}\ \delta\to 0.
$$
\end{lem}

\begin{proof}
Take $\eps_1>0$ and consider
$$
F(x,\delta):=\p\left\{\sigma_x\geq\eps,\ \sup\limits_{t\in[\sigma_x-\eps,\sigma_x]}|w(t)-x|<\delta\right\}\leq\p\left\{\sigma_x\geq\eps\right\}
$$
It is clear that there exists $\rho>0$ such that for all $|x|<\rho$ and all $\delta>0$ $F(x,\delta)<\eps_1$. Let $|x|\geq\rho$ and $\delta<\rho$, estimate
\begin{align*}
F(x,\delta)&=\p\left\{\sigma_{x-\delta\cdot\sgn x}\leq\sigma_x-\eps,\ \sup\limits_{t\in[\sigma_x-\eps,\sigma_x]}|w(t)-x|<\delta\right\} \\
&\leq\p\left\{\sigma_{x-\delta\cdot\sgn x}\leq\sigma_x-\eps\right\}\leq\p\left\{\max\limits_{t\in[0,\eps]}w(t)<\delta\right\}.
\end{align*}
Here we used the strong Markov property of a Wiener process. Hence, there exists $\rho_1\leq\rho$ such that for all $\delta<\rho_1$ and $|x|\geq\rho$
$$
F(x,\delta)\leq\p\{\max\limits_{t\in[0,\eps]}w(t)<\delta\}<\eps_1.
$$
This proves Lemma~\ref{lemm_sup}.
\end{proof}

Let us continue the proof of Lemma~\ref{lemm_conv}. Take $\eps_1>0$ and consider for a fixed $\eps>0$
\begin{align}\label{estim}
\begin{split}
\p\{\tau_n-\tau\geq\eps\}&\leq\p\left\{\tau+\eps\leq\tau_n,\ \sup\limits_{[\tau_n-\eps,\tau_n]}|z_n(t)|<\delta\right\}\\
&+\p\left\{\tau+\eps\leq\tau_n,\ \sup\limits_{[\tau_n-\eps,\tau_n]}|z_n(t)|\geq\delta\right\}\\
&\leq\p\left\{\tau+\eps\leq\tau_n,\ \sup\limits_{[\tau_n-\eps,\tau_n]}|z_n(t)|<\delta\right\}\\
&+\p\left\{\sup\limits_{[\tau_n-\eps,\tau_n]}|z_n(t\wedge\tau_n)-z(t\wedge\tau)|\geq\delta\right\}\\
&\leq\p\left\{\tau+\eps\leq\tau_n,\ \sup\limits_{[\tau_n-\eps,\tau_n]}|z_n(0)+w_n(\langle z_n\rangle_t)|<\delta\right\}\\
&+\p\left\{\sup\limits_{[0,T]}|z_n(t\wedge\tau_n)-z(t\wedge\tau)|\geq\delta\right\}.
\end{split}
\end{align}
Here $\{w_n(t),\ t\geq 0,\ n\geq 1\}$ is a system of Wiener processes such that
$$
z_n(t)=z_n(0)+w_n(\langle z_n(\cdot)\rangle_t),\quad t\in[0,T].
$$
Set
$$
\sigma_n=\inf\{t:\ z_n(0)+w_n(t)=0\}
$$
and
$$
\widetilde{\eps}=p\eps,
$$
where $p$ is defined by~\eqref{rest_char}. {\cob It should be noted that $\langle z_n(\cdot)\rangle_{\tau_n}=\sigma_n$. So, since $\langle z_n(\cdot)\rangle_0=0$,~\eqref{rest_char} implies $p\tau_n\leq\sigma_n$. Moreover, by~\eqref{rest_char}, $$\langle z_n(\cdot)\rangle_{\tau_n-\eps}\leq\langle z_n(\cdot)\rangle_{\tau_n}-p(\tau_n-(\tau_n-\eps))=\sigma_n-\widetilde{\eps}.$$ Now we can estimate
\begin{align*}
\p\Bigg\{\tau&+\eps\leq\tau_n,\ \sup\limits_{[\tau_n-\eps,\tau_n]}|z_n(0)+w_n(\langle z_n\rangle_t)|<\delta\Bigg\}\\
&\leq\p\left\{p\eps\leq p\tau_n,\ \sup\limits_{[\langle z_n(\cdot)\rangle_{\tau_n-\eps},\langle z_n(\cdot)\rangle_{\tau_n}]}|z_n(0)+w_n(t)|<\delta\right\}\\
&\leq\p\left\{\widetilde{\eps}\leq\sigma_n,\ \sup\limits_{[\sigma_n-\widetilde{\eps},\sigma_n]}|z_n(0)+w_n(t)|<\delta\right\}\\
&=\int_{\R}F(x,\delta)(\p\circ z_n(0)^{-1})(dx)\leq\sup\limits_{x\in\R}F(x,\delta).
\end{align*}
}
By Lemma~\ref{lemm_sup} there exists $\delta_1>0$ such that
$$
\sup\limits_{x\in\R}F(x,\delta_1)<\frac{\eps_1}{2}.
$$
Next, choosing $N\in\N$ such that for all $n\geq N$
$$
\p\left\{\sup\limits_{[0,T]}|z_n(t\wedge\tau_n)-z(t\wedge\tau)|\geq\delta_1\right\}<\frac{\eps_1}{2}
$$
and using estimate~\eqref{estim} we obtain
$$
\p\{\tau_n-\tau\geq\eps\}<\eps_1,\quad \mbox{for all}\ n\geq N.
$$
Lemma~\ref{lemm_conv} is proved.
\end{proof}

\begin{lem}\label{lemma_conv_of_moment_of_meeting}
Let $u,v\in[0,1]$, the sequence $\{\tau_{u,v}^n\}_{n\geq 1}$ be defined by~\eqref{f_time_of_meeting_n} and
\begin{equation}\label{f_time_of_meeting}
\tau_{u,v}=\inf\{t:\ y(u,t)=y(v,t)\}\wedge T.
\end{equation}
Then
$$
\lim_{n\to\infty}\tau_{u,v}^n=\tau_{u,v}\quad\mbox{in probability}.
$$
\end{lem}

\begin{proof}
Let $u<v$ and $\eps>0$. Set
\begin{align*}
z_n(t)&=y_n(v,\eps+t)-y_n(u,\eps+t),\quad t\in[0,T-\eps],\\
z(t)&=y(v,\eps+t)-y(u,\eps+t),\quad t\in[0,T-\eps],\\
\tau_n=\inf\{t&:\ z_n(t)=0\}\wedge T,\quad\tau=\inf\{t:\ z(t)=0\}\wedge T.
\end{align*}
Then
$$
\langle z_n(\cdot)\rangle_{t\wedge\tau_n}=\langle y_n(u,\cdot)\rangle_{t\wedge\tau_n}+\langle y_n(v,\cdot)\rangle_{t\wedge\tau_n}\geq 2(t\wedge\tau_n).
$$
By Lemma~\ref{lemm_conv}, $\lim_{n\to\infty}\tau_n=\tau$ in probability. The equations $\tau_n+\eps=\tau_{u,v}^n\vee\eps$ and $\tau+\eps=\tau_{u,v}\vee\eps$ easily imply the assertion of the lemma.
\end{proof}

\begin{lem}\label{lemma_cond_5}
For all $u,v\in[0,1]$ and $t\in[0,T]$, $\langle y(u,\cdot),y(v,\cdot)\rangle_{t\wedge\tau_{u,v}}=0$.
\end{lem}

\begin{proof}
For fixed $u,v\in[0,1]$ denote
$$
M_n^{\eps}(t)=y_n(u,(t\wedge\tau_{u,v}^n)\vee\eps)y_n(v,(t\wedge\tau_{u,v}^n)\vee\eps),\quad t\in[0,T].
$$
First we will show that $M_n^{\eps}(\cdot)$ is an $(\F_{(t\wedge\tau_{u,v}^n)\vee\eps}^{y_n})$-martingale, where $(\F_t^{y_n})$ is generated by $y_n$ (see~\eqref{f_filtration}).

It is well known that $M_n^{\eps}(\cdot)-\langle y_n(u,\cdot),y_n(v,\cdot)\rangle_{(\cdot\wedge\tau_{u,v}^n)\vee\eps}$ is an $(\F_{(t\wedge\tau_{u,v}^n)\vee\eps}^{y_n})$-martingale. Note that $\langle y_n(u,\cdot),y_n(v,\cdot)\rangle_{(\cdot\wedge\tau_{u,v}^n)\vee\eps}$ is also an $(\F_{(t\wedge\tau_{u,v}^n)\vee\eps}^{y_n})$-mar\-tingale because
$$
\langle y_n(u,\cdot),y_n(v,\cdot)\rangle_{(t\wedge\tau_{u,v}^n)\vee\eps}=\langle y_n(u,\cdot),y_n(v,\cdot)\rangle_{\eps}
$$
and it is measurable with respect to $\F_{(0\wedge\tau_{u,v}^n)\vee\eps}^{y_n}$.
Thus $M_n^{\eps}(\cdot)$ is an $(\F_{(t\wedge\tau_{u,v}^n)\vee\eps}^{y_n})$-martingale. Passing to the limit as $n\to\infty$ and using the continuity of the map $(f,a)\to f(\cdot\wedge a)$ from $C[0,T]\times [0,T]$ to $C[0,T]$ and Lemma~\ref{lemma_conv_of_moment_of_meeting} we have
$$
M_n^{\eps}(\cdot)\to y(u,(\cdot\wedge\tau_{u,v})\vee\eps)y(v,(\cdot\wedge\tau_{u,v})\vee\eps)\quad\mbox{in}\ \ C[0,T],\ \ \mbox{in probability}.
$$
By Proposition~9.1.17~\cite{Jacod:2003}, $y(u,(\cdot\wedge\tau_{u,v})\vee\eps)y(v,(\cdot\wedge\tau_{u,v})\vee\eps)$ is a local martingale. Since
$$
y(u,(\cdot\wedge\tau_{u,v})\vee\eps)y(v,(\cdot\wedge\tau_{u,v})\vee\eps)\to y(u,\cdot\wedge\tau_{u,v})y(v,\cdot\wedge\tau_{u,v})\quad\mbox{a.s.},\ \ \eps\to 0,
$$
again, $y(u,\cdot\wedge\tau_{u,v})y(v,\cdot\wedge\tau_{u,v})$ is a local martingale. It proves the lemma.
\end{proof}

To check Condition $(C4)$ we will verify that for each $t\in(0,T]$ the set $\{y(u,t),\ u\in[0,1]\}$ is finite.

\begin{lem}\label{lemma_expect_of_N}
Let $N(t)$ denote the number of distinct points of $\{y(u,t),\ u\in[0,1]\}$, for each $t\in(0,T]$, i.e.
$$
N(t)=|\{y(u,t),\ u\in[0,1]\}|.
$$
Then there exists a constant $C$ such that $\E N(t)\leq\frac{C}{\sqrt{t}}$, $t\in(0,T]$.
\end{lem}

\begin{cor}\label{coroll_step_function}
For all $t\in(0,T]$, $y(\cdot,t)$ is a step function.
\end{cor}

\begin{proof}[Proof of Lemma~\ref{lemma_expect_of_N}]
Denote by $N_n(t)$ the number of distinct points of $\{y_n(u,t),\ u\in[0,1]\}$. First we prove that there exists a constant $C$ which does not depend on $n$ such that $\E N_n(t)\leq\frac{C}{\sqrt{t}}$, $t\in(0,T]$. To show this, it is enough to consider the system $\{x_k^n(t),\ t\in[0,T],\ k\in[n]\}$, which was constructed in Section~\ref{section_finite_system}, and check that $\E \widetilde{N}_n(t)\leq\frac{C}{\sqrt{t}}$, $t\in(0,T]$, where
$$
\widetilde{N}_n(t)=|\{x_k^n(t),\ k\in[n]\}|.
$$
Denote
\begin{align*}
\gamma_1^n(t)&=t,\\
\gamma_k^n(t)&=\inf\{s:\ x_{k-1}^n(s)=x_k^n(s)\}\wedge t,\quad k=2,\ldots,n,\\
\gamma^n(t)&=\sum_{k=1}^n\gamma_k^n(t)
\end{align*}
Let $\{z_k^n(t),\ t\in[0,T],\ k\in[n]\}$ be the system of coalescing Brownian particles which start from $\frac{k}{n}$, $k\in[n]$. Define $\widehat{\gamma}^n(t)$, $t\in(0,T]$, similarly as $\gamma^n(t)$, $t\in(0,T]$, replacing $x_k^n(\cdot)$, $k\in[n]$, with $z_k^n(\cdot)$, $k\in[n]$. By Lemma~\ref{lemma_enequality_for_stop_time} and Lemma~7.1.1~\cite{Dorogovtsev:2007:en}, there exists a constant $C$ such that
$$
\E\gamma^n(t)\leq\E\widehat{\gamma}^n(t)\leq C\sqrt{t},\quad t\in(0,T],\ \ n\geq 1.
$$
Note that
$$
t\widetilde{N}_n(t)\leq\gamma^n(t),\quad t\in(0,T],
$$
so
$$
\E \widetilde{N}_n(t)\leq\frac{C}{\sqrt{t}},\quad t\in(0,T].
$$
It is easy to check that
$$
N(t)\leq\varliminf\limits_{n\to\infty}N_n(t).
$$
Hence, by Fatou's lemma, $\E N(t)\leq\frac{C}{\sqrt{t}}$, $t\in(0,T]$. The lemma is proved.
\end{proof}

\begin{lem}\label{lemma_cond_4}
For every $u\in[0,1]$ and $t\in[0,T]$
$$
\langle y(u,\cdot)\rangle_t=\int_0^t\frac{ds}{m(u,s)},
$$
where $m(u,t)=\Leb\{v:\ \exists s\leq t\ y(v,s)=y(u,s)\}$.
\end{lem}

\begin{proof}
Let $u\in[0,1]$ be fixed and let $\tau_{u,v}^n$ and $\tau_{u,v}$ be defined by~\eqref{f_time_of_meeting_n} and~\eqref{f_time_of_meeting} respectively. Denote
$$
\Xi_u=\left\{\eps\in(0,T):\ \p\{\tau_{u,v}\neq\eps,\ \mbox{for all}\ v\in\Q\cap[0,1]\}=1\right\}.
$$
Note that $(0,T)\setminus\Xi_u$ is countable because it is a subset of the countable set $\left\{\eps\in(0,T):\ \exists\ v\in\Q\cap[0,1]\ \  \p\{\tau_{u,v}=\eps\}>0\right\}.$ By the definition of $m(u,t)$ it is clear that
\begin{equation}\label{f_m_int}
m(u,t)=\int_0^1\I_{\{\tau_{u,v}\leq t\}}dv.
\end{equation}
Similarly,
\begin{equation}\label{f_m_n_int}
m_n(u,t)=\int_0^1\I_{\{\tau_{u,v}^n\leq t\}}dv,
\end{equation}
where $m_n(u,t)$ is defined by~\eqref{f_m_n}. Using Lemma~\ref{lemma_conv_of_moment_of_meeting} and Lemma~4.2~\cite{Kallenberg:2002}, we choose a sequence $\{n'\}$ such that for each $v\in\Q\cap[0,1]$
$$
\tau_{u,v}^{n'}\to\tau_{u,v}\quad\mbox{a.s.},\ \ n'\to\infty.
$$
Fix $\eps\in\Xi_u$ and set
\begin{align*}
\Omega'&=\{\forall v\in\Q\cap[0,1]\ \ \tau_{u,v}^{n'}\to\tau_{u,v}\}\cap\{\forall v\in\Q\cap[0,1]\ \ \tau_{u,v}\neq\eps\}\\
&\cap\{\tau_{u,\cdot},\tau_{u,\cdot}^{n'}\ \mbox{is non-decreasing on}\ [u,1]\\
&\quad\quad\quad\quad\quad\quad\quad\quad\mbox{and non-increasing on}\ [0,u],\ \mbox{for all}\ n'\}\\
&\cap\{y_{n'},y\in D([0,1],C[0,T]),\ \mbox{for all}\ n'\}\cap\{N(\eps)<\infty\}.
\end{align*}
It is clear that $\p\{\Omega'\}=1$. For every $\omega\in\Omega'$ set
$$
R(\omega)=\{\tau_{u,v}(\omega):\ v\in\Q\cap[0,1]\}\cap[\eps,T].
$$
It should be noted that $\eps\not\in R(\omega)$ and $R(\omega)$ is countable. Take $\omega\in\Omega'$ and $t\in[\eps,T]\setminus R(\omega)$ and show that
\begin{equation}\label{f_conv_m}
m_{n'}(u,t,\omega)\to m(u,t,\omega),\quad n'\to\infty.
\end{equation}
To prove it, observe that, since $t\not\in R(\omega)$, for all $v\in\Q\cap[0,1]$
$$
\I_{\{\tau_{u,v}^{n'}(\omega)\leq t\}}\to\I_{\{\tau_{u,v}(\omega)\leq t\}},\quad n'\to\infty.
$$
Next, denote
\begin{align*}
v^+(\omega)&=\sup\{v:\ \tau_{u,v}(\omega)<t\},\\
v^-(\omega)&=\inf\{v:\ \tau_{u,v}(\omega)<t\}.
\end{align*}
and take $v\in(v^-(\omega),v^+(\omega))$. Then $\tau_{u,v}(\omega)<t$. Let $\widetilde{v}$ be a rational point from $(v^-(\omega),v^+(\omega))$ such that $\widetilde{v}>v>u$ or $\widetilde{v}<v<u$. Then $\tau_{u,\widetilde{v}}(\omega)<t$. Since
\begin{align*}
\tau_{u,\widetilde{v}}^{n'}(\omega)&\to\tau_{u,\widetilde{v}}(\omega)\neq t,\\
\I_{\{\tau_{u,v}^{n'}(\omega)\leq t\}}\geq\I_{\{\tau_{u,\widetilde{v}}^{n'}(\omega)\leq t\}}&\to\I_{\{\tau_{u,\widetilde{v}}(\omega)\leq t\}}= \I_{\{\tau_{u,v}(\omega)\leq t\}}=1.
\end{align*}
So, we obtain that for all $v\in(v^-(\omega),v^+(\omega))$
\begin{equation}\label{f_conv_ind}
\I_{\{\tau_{u,v}^{n'}(\omega)\leq t\}}\to\I_{\{\tau_{u,v}(\omega)\leq t\}},\quad n'\to\infty.
\end{equation}
Now consider $v\in[0,1]\setminus[v^-(\omega),v^+(\omega)]$. Note that, by the choice of the set $R(\omega)$ there is no interval $[a,b]\subset[0,1]$ $(a<b)$, such that $\tau_{u,v}(\omega)=t$, $v\in[a,b]$, because the interval contains rational points and we know that for these points the equality is false. So, by the monotonicity of $\tau_{u,\cdot}(\omega)$, we have that
$$
\tau_{u,v}(\omega)>t.
$$
Again, we can show that~\eqref{f_conv_ind} holds for all $v\in[0,1]\setminus[v^-(\omega),v^+(\omega)]$. Thus, by the dominated convergence theorem, \eqref{f_m_int} and~\eqref{f_m_n_int}, we can conclude that~\eqref{f_conv_m} is valid.

Next, note that $\eps\not\in R(\omega)$ so
$$
m_{n'}(u,\eps,\omega)\to m(u,\eps,\omega),\quad n'\to\infty.
$$
Since $y(\cdot,\eps,\omega)$ belongs to $D([0,1],\R)$ and it is a step function (see Corollary~\ref{coroll_step_function}), $m(u,\eps,\omega)>0$.
So,
$$
\frac{1}{m_{n'}(u,\eps,\omega)}\to\frac{1}{m(u,\eps,\omega)}<\infty,\quad n'\to\infty.
$$
Hence there exists a constant $C(u,\eps,\omega)$ such that
$$
\frac{1}{m_{n'}(u,t,\omega)}\leq\frac{1}{m_{n'}(u,\eps,\omega)}\leq C(u,\eps,\omega),\quad t\in[\eps,T]\setminus R(\omega),\ \ n'\geq 1.
$$
{\col Noting that the Lebesgue measure} of $R(\omega)$ equals 0 and using the dominated convergence theorem, we get
\begin{align*}
\sup\limits_{t\in[\eps,T]}&\left|\int_{\eps}^t\frac{ds}{m_{n'}(u,s,\omega)}-\int_{\eps}^t\frac{ds}{m(u,s,\omega)}\right|\\
&\leq\int_{\eps}^T\left|\frac{1}{m_{n'}(u,s,\omega)}-\frac{1}{m(u,s,\omega)}\right|ds\to 0,\quad n'\to\infty.
\end{align*}
Thus
$$
\langle y_{n'}(u,\cdot\vee\eps)\rangle_{\cdot}=\int_{\eps}^{\cdot\vee\eps}\frac{ds}{m_{n'}(u,s,\omega)}\to \int_{\eps}^{\cdot\vee\eps}\frac{ds}{m(u,s,\omega)}\quad\mbox{a.s.}\ \ \mbox{in}\ C[0,T].
$$
Since
$$
y_{n'}(u,\cdot\vee\eps)\to y(u,\cdot\vee\eps)\quad\mbox{a.s.}\ \ \mbox{in}\ C[0,T],
$$
it is easy to see that
$$
\langle y(u,\cdot\vee\eps)\rangle_t=\int_{\eps}^{t\vee\eps}\frac{ds}{m(u,s,\omega)},\quad t\in[0,T].
$$
But on the other hand,
$$
\langle y(u,\cdot\vee\eps)\rangle_t=\langle y(u,\cdot)\rangle_{t\vee\eps}-\langle y(u,\cdot)\rangle_{\eps}
$$
so
$$
\int_{\eps}^{t\vee\eps}\frac{ds}{m(u,s,\omega)}=\langle y(u,\cdot)\rangle_{t\vee\eps}-\langle y(u,\cdot)\rangle_{\eps}, \quad t\in[0,T].
$$
{\cob Since $\langle y(u,\cdot)\rangle_{\cdot}$ is the quadratic variation of a continuous martingale, it is continuous and nondecreasing. Consequently, passing to the limit as $\Xi_u\ni\eps\to 0$, we obtain the proof of the lemma.}
\end{proof}

\subsection{Verification of $(C1)$}\label{section_C1}

In this section we will show that $\{y(u,t),\\ u\in[0,1],\ t\in[0,T]\}$ satisfies Condition $(C1)$, i.e we will prove that $y(u,t)$ has {\col a second moment} for all $u,t$. Since $y(u,\cdot)$ is a local square integrable martingale, it is enough to establish $\E\langle y(u,\cdot)\rangle_t<\infty$. To check this we will prove an estimation of $\p\left\{\frac{1}{m(u,t)}>r\right\}$ and using it we will estimate $\E\frac{1}{m(u,t)}$. The following lemma holds.

\begin{lem}\label{lemma_estim_exspect_m}
For each $\beta\in\left(0,\frac{3}{2}\right)$ the exists a constant $C$ such that for all $u\in[0,1]$
$$
\E\frac{1}{m^{\beta}(u,t)}\leq\frac{C}{\sqrt{t}},\quad t\in(0,T].
$$
\end{lem}

\begin{proof}
Let us estimate the probability
$$
\p\left\{\frac{1}{m(u,t)}>r\right\},\quad r\geq 2,\ \ t\in(0,T].
$$
We can assume, without any restriction of generality, that $u\in\left[0,\frac{1}{2}\right]$. Denote
$$
M(t)=y\left(u+\frac{1}{r},t\right)-y(u,t)
$$
and
$$
A_t=\left\{\frac{1}{m(u,t)}>r\right\}.
$$
Since $M(\cdot)$ is a continuous martingale (see Proposition~\ref{prop_prolong}), $M(\cdot)$ is a continuous local square integrable martingale. {\col To calculate the quadratic variation of $M(\cdot)$ we have}
$$
\langle M(\cdot)\rangle_t=\left\langle y\left(u+\frac{1}{r},\cdot\right)\right\rangle_t+\langle y(u,\cdot)\rangle_t- 2\left\langle y\left(u+\frac{1}{r},\cdot\right),y(u,\cdot)\right\rangle_t.
$$
Note that by Lemma~\ref{lemma_cond_5},
$$
\left\langle y\left(u+\frac{1}{r},\cdot\right),y(u,\cdot)\right\rangle_t\I_{\{M(t)>0\}}=0,\quad t\in[0,T].
$$
Take $\omega\in A_t$ then $\omega\in\{M(t)>0\}$ because $y\left(u+\frac{1}{r},\cdot,\omega\right)$ and $y(u,\cdot,\omega)$ do not meet by time $t$. Hence
\begin{align*}
\langle M(\cdot)\rangle_t(\omega)&=\left\langle y\left(u+\frac{1}{r},\cdot\right)\right\rangle_t(\omega)+\langle y(u,\cdot)\rangle_t(\omega)\\
&\geq\langle y(u,\cdot)\rangle_t(\omega)=\int_0^t\frac{ds}{m(u,s,\omega)}\geq rt.
\end{align*}
Next, since $M(\cdot)$ {\col is a continuous martingale,} there exists a Wiener process $w(t),\ t\geq 0$, such that
$$
M(t)=\frac{1}{r}+w\left(\langle M(\cdot)\rangle_t\right).
$$
Set
$$
\tau=\inf\{t:\ M(t)=0\}\wedge T,\quad \sigma=\inf\{t:\ \frac{1}{r}+w(t)=0\}.
$$
As in the proof of Lemma~\ref{lemma_enequality_for_stop_time} we can prove that
$$
\langle M(\cdot)\rangle_{\tau}\leq\sigma.
$$
Note that if $\omega\in A_t$ then $\tau(\omega)>t$ and hence, by the last inequality $\sigma(\omega)\geq\langle M(\cdot)\rangle_{\tau(\omega)}(\omega)\geq \langle M(\cdot)\rangle_{t}(\omega)\geq rt$.
Now we are ready to estimate the probability of $A_t$. So, {\cob
\begin{align*}
\p\{A_t\}&=\p\{A_t,\ M(t)>0\}=\p\{A_t,\ \tau>t\}\leq\p\{A_t,\ \sigma\geq rt\}\\
&\leq\p\{\sigma\geq rt\}=\p\left\{\max\limits_{s\in[0,rt]}w(s)<\frac{1}{r}\right\}\leq \p\left\{\max\limits_{s\in[0,1]}w(s)<\frac{1}{\sqrt{tr^3}}\right\}\\
&\leq\frac{2}{\sqrt{2\pi}}\int_0^{\frac{1}{\sqrt{tr^3}}}e^{-\frac{x^2}{2}}dx\leq\frac{2}{\sqrt{2\pi}}\frac{1}{\sqrt{tr^3}}.
\end{align*}
}
Thus,
{\cob
\begin{align*}
\E\frac{1}{m^{\beta}(u,t)}&=\E\int_0^{\infty}\I_{\left\{\frac{1}{m^{\beta}(u,t)}>r\right\}}dr= \int_0^{\infty}\p\left\{\frac{1}{m(u,t)}>r^{\frac{1}{\beta}}\right\}dr\\
&\leq 2+\int_2^{\infty}\frac{2}{\sqrt{2\pi}}\frac{1}{\sqrt{t}r^{\frac{3}{2\beta}}}dr\leq\frac{C}{\sqrt{t}}.
\end{align*}
}
The lemma is proved.
\end{proof}

\begin{lem}\label{lemma_estim_expect_char}
There exists a constant $C$ such that for all $u\in[0,1]$
$$
\E\int_0^t\frac{ds}{m(u,s)}\leq C\sqrt{t},\quad t\in[0,T].
$$
\end{lem}

\begin{proof}
The assertion of the lemma immediately follows from Lemma~\ref{lemma_estim_exspect_m}.
\end{proof}

This lemma {\col allows us to obtain} the following result.

\begin{lem}\label{lemma_expect_of_y2}
There exists a constant $C$ such that for all $u\in[0,1]$
$$
\E (y(u,t)-u)^2\leq C\sqrt{t},\quad t\in[0,T].
$$
Moreover, $y(u,\cdot)$ is a continuous square integrable martingale.
\end{lem}

\begin{rmk}\label{remark_prop_y}
Since we did not use the fact that $\{y(u,t),\ u\in[0,1],\ t\in[0,T]\}$ was a limit of a finite system of particles, we claim that {\col all the results of this subsection are valid for every random element} $\{y(u,t),\ u\in[0,1],\ t\in[0,T]\}$ of $D([0,1],C[0,T])$ which satisfies conditions $(C2)$-$(C5)$ and
\begin{enumerate}
\item[(C1')] for all $u\in[0,1]$ the process $y(u,\cdot)$ is a continuous martingale with respect to the filtration
$$
\mathcal{F}_t=\sigma(y(u,s),\ u\in[0,1],\ s\leq t),\quad t\in[0,T].
$$
\end{enumerate}

\end{rmk}

\section{Stochastic integral with respect to the constructed flow and an analog of Ito's formula}\label{section_Ito_formula}

Hereafter we will suppose that $\{y(u,t),\ u\in[0,1],\ t\in[0,T]\}$ is a random element of $D([0,1],C[0,T])$ which satisfies conditions $(C1)$--$(C5)$. For such a flow we will construct a stochastic integral
$
\int_0^1\int_0^t\varphi(y(u,s))dy(u,s)du
$
and using the constructed integral we obtain an analog of Ito's formula. First let us establish a property of $\{y(u,t),\ u\in[0,1],\ t\in[0,T]\}$.

\begin{lem}\label{lemma_step_funct}
For all $t\in(0,T]$ the function $y(u,t),\ u\in[0,1]$, is a step function in $D([0,1],\R)$. Moreover $\p\{\mbox{for all}\ u,v\in[0,1],\ \mbox{if}\ y(u,t)=y(v,t)\ \mbox{then}\ y(u,t+\cdot)=y(v,t+\cdot)\}=1$.
\end{lem}

Note that earlier this property followed from the fact that $\{y(u,t),\ u\in[0,1],\ t\in[0,T]\}$ was approximated by a finite particle system. But in fact this property follows from $(C1)$-$(C5)$.

\begin{proof}[Proof of Lemma~\ref{lemma_step_funct}]
This is the same proof as the one of a similar result for a coalescing Brownian motion (see e.g.~Section~7.1~\cite{Dorogovtsev:2007:en}).
\end{proof}

Let $\M$ denote the space of continuous square integrable {\col martingales on $[0,T]$ with respect to} the filtration defined by~\eqref{f_filtration_y} and let
$$
(M,N)=\E M(T)N(T),\quad M,N\in\M
$$
be an inner product on $\M$. It is well known that $\M$ is a Hilbert space.

Consider for $n\geq 1$ a partition $0=u_0^n<\ldots<u_n^n=1$. Let $v_k^n\in[u_{k-1}^n,u_k^n]$, $k\in[n]$, $\lambda_n=\max\limits_{k\in[n]}\Delta u_k^n$, where $\Delta u_k^n=u_k^n-u_{k-1}^n$, $k\in[n]$, and let $\varphi$ is a bounded piecewise continuous function from $\R$ to $\R$. Set
$$
M_n(t)=\sum_{k=1}^n\int_0^t\varphi(y(v_k^n,s))dy(v_k^n,s)\Delta u_k^n,\quad t\in[0,T].
$$
Note that $M_n(\cdot)$ belongs to $\M$. The following proposition holds.

\begin{prp}\label{prop_fundam}
The sequence $\{M_n(\cdot)\}_{n\geq 1}$ is convergent in $\M$ as $\lambda_n\to 0$ and its limit does not depend on the choice of the partition $u_k^n$, $k\in[n]$, and the points $v_k^n$, $k\in[n]$.
\end{prp}

\begin{proof}
To prove the proposition it is enough to check that $\{M_n(\cdot)\}_{n\geq 1}$ is {\col a Cauchy sequence}. So, consider
\begin{equation}\label{f_inner_prod}
\begin{split}
(M_n(\cdot)&,M_p(\cdot))=\E M_n(T)M_p(T)=\E\langle M_n(\cdot),M_p(\cdot)\rangle_T\\
&=\E\sum_{k=1}^n\sum_{l=1}^p\int_0^T\frac{\varphi(y(v_k^n,s))\varphi(y(v_l^p,s))}{\sqrt{m(v_k^n,s)m(v_l^p,s)}} \I_{\{\tau_{v_k^n,v_l^p}\leq s\}}\Delta u_k^n\Delta u_l^pds\\
&=\E\int_0^T\sum_{k=1}^n\varphi^2(y(v_k^n,s))\Delta u_k^n \frac{\sum_{l=1}^p\I_{\{\tau_{v_k^n,v_l^p}\leq s\}}\Delta u_l^p}{m(v_k^n,s)}ds
\end{split}
\end{equation}
\begin{align*}
&=\E\int_0^T\sum_{k=1}^n\varphi^2(y(v_k^n,s))\Delta u_k^nds\\
&+\E\int_0^T\sum_{k=1}^n\varphi^2(y(v_k^n,s))\Delta u_k^n\left[\frac{\sum_{l=1}^p\I_{\{\tau_{v_k^n,v_l^p}\leq s\}}\Delta u_l^p}{m(v_k^n,s)}-1\right]ds\\
&=I_{n,p}^1+I_{n,p}^2.
\end{align*}
By the dominated convergence theorem
$$
I_{n,p}^1\to\E\int_0^T\int_0^1\varphi^2(y(u,s))duds,\quad \lambda_n,\lambda_p\to 0.
$$
Estimate the second term of the right hand side of~\eqref{f_inner_prod}
\begin{align*}
|I_{n,p}^2|&\leq\E\int_0^T\sum_{k=1}^n\varphi^2(y(v_k^n,s))\Delta u_k^n\left|\frac{\sum_{l=1}^p\I_{\{\tau_{v_k^n,v_l^p}\leq s\}}\Delta u_l^p}{m(v_k^n,s)}-1\right|ds\\
&\leq\|\varphi^2\|\int_0^T\sum_{k=1}^n\Delta u_k^n\E\left|\frac{\sum_{l=1}^p\I_{\{\tau_{v_k^n,v_l^p}\leq s\}}\Delta u_l^p}{m(v_k^n,s)}-1\right|ds\\
&\leq\|\varphi^2\|\int_0^T\sup\limits_{u\in[0,1]}\E\frac{\left|\sum_{l=1}^p\I_{\{\tau_{u,v_l^p}\leq s\}}\Delta u_l^p-m(u,s)\right|}{m(u,s)}ds\\
&\leq\|\varphi^2\|\int_0^T\sup\limits_{u\in[0,1]}\E\frac{2\lambda_p}{m(u,s)}ds\leq 2\lambda_p\|\varphi^2\| C\sqrt{T}.
\end{align*}
Here the latter inequality follows from Lemma~\ref{lemma_estim_exspect_m}. Hence $I_{n,p}\to 0$ as $\lambda_n,\lambda_p\to 0$. {\col Thus  $\{M_n(\cdot)\}_{n\geq 1}$ is a Cauchy sequence} in $\M$ and therefore $\{M_n(\cdot)\}_{n\geq 1}$ is convergent. By the method of mixing of two sequences, it is easy to show that the limit does not depend on the choice of the permutation. The proposition is proved.
\end{proof}

Denote
$$
\int_0^1\int_0^{\cdot}\varphi(y(u,s))dy(u,s)du=\lim_{\lambda_n\to 0}M_n(\cdot)\quad\mbox{in}\ \M.
$$

\begin{prp}
For each bounded piecewise continuous function $\varphi$, $\int_0^1\int_0^{t}\varphi(y(u,s))dy(u,s)du$, $t\in[0,T]$, is a continuous square integrable $(\F_t)$-martingale with the quadratic variation
\begin{equation}\label{f_charact_of_int}
\left\langle\int_0^1\int_0^{\cdot}\varphi(y(u,s))dy(u,s)du\right\rangle_t= \int_0^1\int_0^t\varphi^2(y(u,s))dsdu,\quad t\in[0,T].
\end{equation}
\end{prp}

\begin{proof}
The martingale property of $\int_0^1\int_0^{\cdot}\varphi(y(u,s))dy(u,s)du$ follows from its construction. Let us check~\eqref{f_charact_of_int}. As in the proof of Proposition~\ref{prop_fundam} we can show that for each $t\in[0,T]$
$$
\langle M_n(\cdot)\rangle_t\to\int_0^1\int_0^t\varphi^2(y(u,s))dsdu\quad\mbox{in}\ L_1
$$
Next, take $r\leq t$ and consider
$$
\E(M_n^2(t)-\langle M_n(\cdot)\rangle_t|\F_r)=M_n^2(r)-\langle M_n(\cdot)\rangle_r.
$$
By the dominated convergence theorem, for the conditional expectations one has
\begin{align*}
&\E\left(\left.\left(\int_0^1\int_0^t\varphi(y(u,s))dy(u,s)du\right)^2- \int_0^1\int_0^t\varphi^2(y(u,s))dsdu\right|\F_r\right)\\
&=\left(\int_0^1\int_0^r\varphi(y(u,s))dy(u,s)du\right)^2-\int_0^1\int_0^r\varphi^2(y(u,s))dsdu.
\end{align*}
Hence~\eqref{f_charact_of_int} is valid. It completes the proof of the proposition.
\end{proof}

Now we are ready to establish an analog of Ito's formula.

\begin{thm}\label{theorem_Ito_formula}
For each twice continuously differentiable function $\varphi:\R\to\R$ having bounded derivatives we have
\begin{align*}
\int_0^1\varphi(y(u,t))du&=\int_0^1\varphi(u)du+\int_0^1\int_0^t\dot{\varphi}(y(u,s))dy(u,s)du\\ &+\frac{1}{2}\int_0^t\int_0^1\frac{\ddot{\varphi}(y(u,s))}{m(u,s)}duds,\quad t\in[0,T].
\end{align*}
\end{thm}

\begin{proof}
Let $\eps>0$ and $0=u_0^n<\ldots<u_n^n=1$. By Ito's formula,
\begin{align*}
\sum_{k=1}^n\varphi(y(u_k^n,t))\Delta u_k^n&=\sum_{k=1}^n\varphi(y(u_k^n,\eps))\Delta u_k^n\\
&+\sum_{k=1}^n\int_{\eps}^t\dot{\varphi}(y(u_k^n,s))dy(u_k^n,s)\Delta u_k^n\\ &+\frac{1}{2}\sum_{k=1}^n\int_{\eps}^t\frac{\ddot{\varphi}(y(u_k^n,s))}{m(u_k^n,s)}ds\Delta u_k^n=I_n^1+I_n^2+I_n^3.
\end{align*}
Note that
$$
I_n^1\to\int_0^1\varphi(y(u,\eps))du\quad\mbox{a.s.},\ \ \lambda_n\to 0,
$$
and
$$
I_n^2\to M(t)-M(\eps)\quad\mbox{in}\ L_2, \ \ \lambda_n\to 0,
$$
where
$$
M(t)=\int_0^1\int_0^t\dot{\varphi}(y(u,s))dy(u,s)du,\quad t\in[0,T].
$$
Hence these sequences converge in probability. Denote
$$
S_n(s)=\sum_{k=1}^n\frac{\ddot{\varphi}(y(u_k^n,s))}{m(u_k^n,s)}\Delta u_k^n,\quad s\in[\eps,t]
$$
By Lemma~\ref{lemma_step_funct} almost surely $m(\cdot,s)$ is a step function with a finite number of discontinuity points  and $\sup\limits_{u\in[0,1]}\frac{1}{m(u,s)}<\infty$, for all $s\in[\eps,t]$, so
$$
S_n(s)\to\int_0^1\frac{\ddot{\varphi}(y(u,s))}{m(u,s)}du\quad\mbox{a.s.},\ \ \lambda_n\to 0,\ \ s\in[\eps,t].
$$
Next by the monotonicity of $m(u,s),\ s\in[\eps,t]$, for all $u\in[0,1]$, we have
$$
S_n(s)\leq\|\ddot{\varphi}\|\sup\limits_{u\in[0,1]}\frac{1}{m(u,\eps)}<\infty,\quad s\in[\eps,t],\ \ a.s.
$$
Thus by the dominated convergence theorem,
$$
\int_{\eps}^t S_n(s)ds\to\int_{\eps}^t\int_0^1\frac{\ddot{\varphi}(y(u,s))}{m(u,s)}duds\quad\mbox{a.s.},\ \ \lambda_n\to 0.
$$
So,
$$
I_n^3\to\int_{\eps}^t\int_0^1\frac{\ddot{\varphi}(y(u,s))}{m(u,s)}duds\quad\mbox{in probability},\ \ \lambda_n\to 0.
$$
Hence
\begin{align*}
\int_0^1\varphi(y(u,t))du&=\int_0^1\varphi(y(u,\eps))du+M(t)-M(\eps)\\
&+\frac{1}{2}\int_{\eps}^t\int_0^1\frac{\ddot{\varphi}(y(u,s))}{m(u,s)}duds.
\end{align*}
{\col
Note that the integral $\int_0^t\int_0^1\frac{\ddot{\varphi}(y(u,s))}{m(u,s)}duds$ exists a.s., since
$$
\E\int_0^t\int_0^1\frac{|\ddot{\varphi}(y(u,s))|}{m(u,s)}duds\leq \sup_{x\in\R}|\ddot{\varphi}(x)|\E\int_0^t\int_0^1\frac{duds}{m(u,s)}<\infty,
$$
by Lemma~\ref{lemma_estim_expect_char}. So, passing to the limit as $\eps\to 0$ and using the dominated convenience theorem, we obtain
\begin{align*}
\lim_{\eps\to 0}\frac{1}{2}\int_{\eps}^t\int_0^1\frac{\ddot{\varphi}(y(u,s))}{m(u,s)}duds&=-\int_0^1\varphi(u)du\\
&-M(t)+\int_0^1\varphi(y(u,t))du,
\end{align*}
where $\lim$ we understand as the limit almost surely. The theorem is proved.
}
\end{proof}

\section{Total local time}\label{section_local_time}

In this section we define a local time for the constructed flow and prove its existence.

\begin{dfn}\label{definition_local_time}
A random process $\{L(a,t),\ a\in\R,\ t\in[0,T]\}$ is said to be the local time for the process $\{y(u,t),\ u\in[0,1],\ t\in[0,T]\}$ if
\begin{enumerate}
\item[(a)] $(a,t)\to L(a,t)$ is a continuous map a.s.;

\item[(b)] for every continuous function $f$ with compact support
$$
\int_0^1\int_0^{\tau(u)\wedge t}f(y(u,s))ds=2\int_{-\infty}^{+\infty}f(a)L(a,t)da.
$$
where the integral in the left hand side is defined by A.~A.~Dorogovtsev (see e.g.~\cite{Dorogovtsev:2007:en} or Appendix~\ref{app_integral})
\end{enumerate}
\end{dfn}

\begin{thm}\label{theorem_local_time}
The local time for the flow $\{y(u,t),\ u\in[0,1],\ t\in[0,T]\}$ exists. Moreover
\begin{align}\label{f_local_time}
\begin{split}
L(a,t)&=\int_0^1(y(u,t)-a)^+du-\int_0^1(u-a)^+du\\
&-\int_0^1\int_0^t\I_{(a,+\infty)}(y(u,s))dy(u,s)du.
\end{split}
\end{align}
\end{thm}

\begin{proof}
Let $g_n$ be a continuous positive function on $\R$ such that its support is contained in $\left(-\frac{1}{n}+a,\frac{1}{n}+a\right)$, $g_n(a+u)=g_n(a-u)$, $u\in\R$, and
$$
\int_{-\infty}^{+\infty}g_n(u)du=1.
$$
Set
$$
f_n(u)=\int_{-\infty}^udp\int_{-\infty}^pg_n(q)dq.
$$
By Ito's formula (see Theorem~\ref{theorem_Ito_formula}) and Corollary~\ref{lemma_integral},
\begin{align*}
\int_0^1f_n(&y(u,t))du-\int_0^1f_n(u)du-\int_0^1\int_0^t\dot{f}_n(y(u,s))dy(u,s)du\\
&=\frac{1}{2}\int_0^1\int_0^t\frac{\ddot{f}_n(y(u,s))}{m(u,s)}dsdu=\frac{1}{2}\int_0^1\int_0^{\tau(u)\wedge t}\ddot{f}_n(y(u,s))ds.
\end{align*}
If the local time exists then passing to the limit in previous expression we deduce that it should have the form~\eqref{f_local_time}.

{\cob In what follows we will show that} the family of random variables $\{L(a,t),\ a\in\R,\ t\in[0,T]\}$ satisfies properties $(a)$ and $(b)$ of Definition~\ref{definition_local_time}. Since $(y(u,t)-a)^+-(u-a)^+$ is continuous in $(a,t)$ with probability 1, for all $u\in[0,1]$,
$$
\int_0^1(y(u,t)-a)^+du-\int_0^1(u-a)^+du
$$
is also continuous in $(a,t)$ a.s. Next, let us show that the $C[0,T]$-valued process
$$
\xi(a,\cdot)=\int_0^1\int_0^{\cdot}\I_{(a,+\infty)}(y(u,s))dy(u,s)du,\quad a\in\R,
$$
has a continuous modification. Note that for $a<b$
$$
\xi(a,t)-\xi(b,t)=\int_0^1\int_0^{t}\I_{(a,b]}(y(u,s))dy(u,s)du,\quad t\in[0,T],
$$
is a continuous square integrable martingale with the quadratic variation
$$
\langle\xi(a,\cdot)-\xi(b,\cdot)\rangle_t=\int_0^1\int_0^{t}\I_{(a,b]}(y(u,s))dsdu,\quad t\in[0,T].
$$
By Theorem~3.3.1~\cite{Watanabe:1981:en}
\begin{align*}
\E\sup\limits_{t\in[0,T]}&\left|\xi(a,t)-\xi(b,t)\right|^4\leq C\E\langle\xi(a,\cdot)-\xi(b,\cdot)\rangle_T^2\\
&=C\E\left(\int_0^1\int_0^{T}\I_{(a,b]}(y(u,s))dsdu\right)^2.
\end{align*}

Let $\psi:\R\to\R$ be a twice continuously differentiable function with bounded derivatives such that
$$
\ddot{\psi}(u)\geq2\I_{(a,b]}(u),\quad |\dot{\psi}(u)|\leq 2(b-a),\quad u\in\R.
$$
Then, by Theorem~\ref{theorem_Ito_formula}
\begin{align*}
\int_0^1\int_0^{T}\I_{(a,b]}&(y(u,s))dsdu\leq\frac{1}{2}\int_0^1\int_0^{T}\ddot{\psi}(y(u,s))dsdu\\
&\leq\frac{1}{2}\int_0^1\int_0^{T}\frac{\ddot{\psi}(y(u,s))}{m(u,s)}dsdu=\int_0^1\psi(y(u,T))du-\int_0^1\psi(u)du\\
&-\int_0^1\int_0^T\dot{\psi}(y(u,s))dy(u,s)du\leq\|\dot{\psi}\|\int_0^1|y(u,T)-u|du\\
&+\left|\int_0^1\int_0^T\dot{\psi}(y(u,s))dy(u,s)du\right|.
\end{align*}
Next estimate
\begin{align*}
&\E\left(\int_0^1\int_0^T\dot{\psi}(y(u,s))dy(u,s)du\right)^2=\\
&=\E\int_0^1\int_0^T\dot{\psi}^2(y(u,s))dsdu\leq 4T(b-a)^2.
\end{align*}
Consequently, by Lemma~\ref{lemma_expect_of_y2}
\begin{align*}
\E\sup\limits_{t\in[0,T]}&\left|\xi(a,t)-\xi(b,t)\right|^4\leq2\|\dot{\psi}\|^2\E\left(\int_0^1|y(u,T)-u|du\right)^2\\
&+2\E\left(\int_0^1\int_0^T\dot{\psi}(y(u,s))dy(u,s)du\right)^2\\
&\leq 8(b-a)^2\E\int_0^1(y(u,T)-u)^2du\\
&+8T(b-a)^2\leq C_1(b-a)^2.
\end{align*}
So, $\xi(a,\cdot)$ has a continuous modification. Further we will consider only this modification. Hence $L(a,t),\ a\in\R,\ t\in[0,T]$, is a continuous process. Let us check Condition $(b)$ of Definition~\ref{definition_local_time}. Let $f$ be a continuous function on $\R$ with a compact support. Set
$$
F(x)=\int_{-\infty}^{+\infty}f(a)(x-a)^+da,\quad x\in\R.
$$
It is clear that $F$ is a twice continuously differentiable function with
$$
\dot{F}(x)=\int_{-\infty}^xf(a)da,\quad
\ddot{F}(x)=f(x),\quad x\in\R.
$$
By Ito's formula (see Theorem~\ref{theorem_Ito_formula}),
\begin{align*}
\int_0^1F(&y(u,t))du-\int_0^1F(u)du-\int_0^1\int_0^t\dot{F}(y(u,s))dy(u,s)ds\\
&=\frac{1}{2}\int_0^1\int_0^t\frac{\ddot{F}(y(u,s))}{m(u,s)}dsdu=\frac{1}{2}\int_0^1\int_0^{\tau(u)\wedge t}f(y(u,s))ds.
\end{align*}
Rewrite the left hand side of the previous relation
\begin{align*}
\int_{-\infty}^{+\infty}f(&a)\left[\int_0^1(y(u,t)-a)^+du-\int_0^1(u-a)^+du\right]da\\
&-\int_0^1\int_0^t\int_{-\infty}^{+\infty}f(a)\I_{(a,+\infty)}(y(u,s))dady(u,s)du.
\end{align*}
If we could rearrange the order of integration in the last integral then we would obtain
\begin{align*}
&\int_{-\infty}^{+\infty}f(a)\left[\int_0^1(y(u,t)-a)^+du-\int_0^1(u-a)^+du\right.\\
&\left.-\int_0^1\int_0^t\I_{(a,+\infty)}(y(u,s))dy(u,s)du\right]da=\frac{1}{2}\int_0^1\int_0^{\tau(u)\wedge t}f(y(u,s))ds.
\end{align*}
Let us show that we can rearrange the order of integration. First prove that
$$
I(t)=\int_{-\infty}^{+\infty}\left(f(a)\int_0^1\int_0^t\I_{(a,+\infty)}(y(u,s))dy(u,s)du\right)da,\quad t\in[0,T],
$$
has a continuous modification. To check it we estimate
\begin{align*}
\E\int_{-\infty}^{+\infty}|&f(a)|\max\limits_{t\in[0,T]}\left|\int_0^1\int_0^t \I_{(a,+\infty)}(y(u,s))dy(u,s)du\right|da\\
&\leq\int_{-\infty}^{+\infty}|f(a)|\left(\E\max\limits_{t\in[0,T]}\left|\int_0^1\int_0^t \I_{(a,+\infty)}(y(u,s))dy(u,s)du\right|^2\right)^{\frac{1}{2}}da\\
&\leq C\int_{-\infty}^{+\infty}|f(a)|\left(\E\int_0^1\int_0^T \I_{(a,+\infty)}(y(u,s))dsdu\right)^{\frac{1}{2}}da\\
&\leq CT\int_{-\infty}^{+\infty}|f(a)|da < +\infty.
\end{align*}
Since $\xi(a,t),\ a\in\R,\ t\in[0,T],$ is continuous a.s., $\xi(a,\cdot)$ is continuous with probability 1, for all $a\in\R$. From the last estimation it easily follows that $I(t),\ t\in[0,T]$, is continuous a.s.

Next, let $\supp f\subseteq [-M,M]$,
$$
-M=a_0\leq \widetilde{a}_1\leq a_1\leq\widetilde{a}_2\leq\ldots\leq a_n=M
$$
and $\lambda_n=\max\limits_{k\in[n]}(a_k-a_{k-1})$. Consider for $t\in[0,T]$
\begin{align*}
I_n(t)&=\sum_{k=1}^nf(\widetilde{a}_k)\int_0^1\int_0^t\I_{(\widetilde{a}_k,+\infty)}(y(u,s))dy(u,s)du\Delta a_k\\
&=\int_0^1\int_0^t\sum_{k=1}^nf(\widetilde{a}_k)\I_{(\widetilde{a}_k,+\infty)}(y(u,s))\Delta a_kdy(u,s)du.
\end{align*}
Note that for a fixed $t\in[0,T]$
$$
I_n(t) \to \int_{-\infty}^{+\infty}\left(f(a)\int_0^1\int_0^t\I_{(a,+\infty)}(y(u,s))dy(u,s)du\right)da\quad \lambda_n\to 0,\ \ \mbox{a.s.}
$$

On the other hand,
\begin{align*}
&\left\langle\int_0^1\int_0^{\cdot} \left[\sum_{k=1}^nf(\widetilde{a}_k)\I_{(\widetilde{a}_k,+\infty)}(y(u,s))\Delta a_k\right.\right.\\
&\left.\left.-\int_{-\infty}^{+\infty}f(a)\I_{(a,+\infty)}(y(u,s))da\right]dy(u,s)du\right\rangle_t\\
&=\int_0^1\int_0^t \left[\sum_{k=1}^nf(\widetilde{a}_k)\I_{(\widetilde{a}_k,+\infty)}(y(u,s))\Delta a_k\right.\\
&\left.-\int_{-\infty}^{+\infty}f(a)\I_{(a,+\infty)}(y(u,s))da\right]^2dsdu\to 0,\quad\lambda_n\to 0,\ \ \mbox{a.s.},
\end{align*}
by the dominated convergence theorem. So,
$$
I_n(t)\to\int_0^1\int_0^t\left(\int_{-\infty}^{+\infty}f(a)\I_{(a,+\infty)}(y(u,s))da\right)dy(u,s)du,\quad\lambda_n\to 0,\ \ \mbox{in}\ L_2.
$$
Hence, by the continuity of the integrals in $t$ we have
\begin{align*}
&\int_{-\infty}^{+\infty}\left(\int_0^1\int_0^{\cdot}f(a)\I_{(a,+\infty)}(y(u,s))dy(u,s)du\right)da\\
&=\int_0^1\int_0^{\cdot}\left(\int_{-\infty}^{+\infty}f(a)\I_{(a,+\infty)}(y(u,s))da\right)dy(u,s)du.
\end{align*}
This proves the theorem.
\end{proof}

\appendix

\section{Some results about tightness in space $D([0,1],C(0,T])$}
Here we will show that the tightness of a system of probability measures $\{P_n\}_{n\geq 1}$ in the space $D([0,1],C(0,T])$ is equivalent to the tightness of the set of probability measures $\{P_n\circ\pi_{\eps}^{-1}\}_{n\geq 1}$ in $D([0,1],C[\eps,T])$ for all $\eps>0$, where $\pi_{\eps}$ is a restriction map on $D([0,1],C[\eps,T])$, which will be defined later.

Denote by $C(0,T]$ the space of continuous functions from $(0,T]$ to $\R$ with the metric
$$
\rho(f,g)=\sum_{k=1}^{\infty}\frac{1}{2^k}\left(\sup\limits_{t\in\left[\frac{1}{k},T\right]}|f(t)-g(t)|\wedge 1\right).
$$
For a function $f:[0,1]\times(0,T]\to\R$ and $\eps>0$ set
$$
(\pi_{\eps}f)(u,t)=\pi_{\eps}f(u,t)=f(u,t),\quad u\in[0,1],\ \ t\in[\eps,T].
$$
Note that $\pi_{\eps}f:[0,1]\times[\eps,T]\to\R$.

\begin{lem}\label{lemma_funct_f}
A function $f:[0,1]\times(0,T]\to\R$ belongs to $D([0,1],C(0,T])$ if and only if for every $\eps>0$ $\pi_{\eps}f\in D([0,1],C[\eps,T])$.
\end{lem}

The proof is a standard technical exercise.

\begin{lem}\label{lemma_convergence_in_D}
Let $\{f_n\}_{n\geq 1}$ be a sequence of functions in $D([0,1],C(0,T])$.
\begin{enumerate}
\item[(i)] If the sequence $\{f_n\}_{n\geq 1}$ converges to a function $f$ in $D([0,1],C(0,T])$ then for every $\eps>0$ $\{\pi_{\eps}f_n\}_{n\geq 1}$ converges to $\pi_{\eps}f$ in $D([0,1],C[\eps,T])$.

\item[(ii)] If for every $\eps>0$ $\{\pi_{\eps}f_n\}_{n\geq 1}$ converges to a function $g_{\eps}$ in $D([0,1],C[\eps,T])$ then the function
\begin{equation}\label{f_function_f}
f(u,t)=g_{\eps}(u,t),\quad u\in[0,1],\ \ t\in(0,T], \ \ \eps\leq t,
\end{equation}
is well defined and $\{f_n\}_{n\geq 1}$ converges to $f$ in $D([0,1],C(0,T])$.
\end{enumerate}
\end{lem}

\begin{proof}
Let $d_0$ and $d_{\eps}$ be the Skorohod metrics in $D([0,1],C(0,T])$ and $D([0,1],C[\eps,T])$ respectively. The inequality
$$
d_{\eps}(\pi_{\eps}f,\pi_{\eps}g)\leq Cd_0(f,g),\quad f,g\in D([0,1],C(0,T]),
$$
implies the assertion of the first part of the lemma.

To prove the second part, we first recall the definition of the Skorohod metric. Let $\Lambda$ denote a set of strictly increasing Lipschitz continuous functions from $[0,1]$ onto $[0,1]$. Set for $\lambda\in\Lambda$
$$
\gamma(\lambda)=\sup\limits_{0\leq s<t\leq 1}\left|\ln{\frac{\lambda(t)-\lambda(s)}{t-s}}\right|
$$
Then the metric on $D([0,1],C(0,T])$ is defined by
$$
d_0(f,g)=\inf\limits_{\lambda\in\Lambda}\left[\gamma(\lambda)\vee\sup\limits_{u\in[0,1]}\rho(f(\lambda(u),\cdot),g(u,\cdot))\wedge 1\right].
$$
The metric $d_{\eps}$ is defined similarly as $d_0$ replacing $\rho$ with $\rho_{\eps}$, where $\rho_{\eps}$ is uniform distance on $C[\eps,T]$.

Let $f$ be defined by~\eqref{f_function_f}. Prove that $f$ is well defined. Take $t\in[\eps_1\vee\eps_2,T]$. Then for each $n\geq 1$
$$
\pi_{\eps_1}f_n(\cdot,t)=\pi_{\eps_2}f_n(\cdot,t).
$$
Since $\pi_{\eps_i}f_n(\cdot,t)\to g_{\eps_i}(\cdot,t)$ in $D([0,1],\R)$, $i=1,2$, $g_{\eps_1}(\cdot,t)=g_{\eps_2}(\cdot,t)$. So, $f$ is well defined. Note that for each $\eps>0$, $\pi_{\eps}f=g_{\eps}$ so by Lemma~\ref{lemma_funct_f}, $f\in D([0,1],C(0,T])$.

Next, fix $\delta>0$ and take $p\in\N$ such that $\frac{1}{2^p}<\frac{\delta}{2}$. Since
$$
d_{\frac{1}{p}}\left(\pi_{\frac{1}{p}}f_n,\pi_{\frac{1}{p}}f\right)\to 0,\quad\mbox{as}\ n\to\infty,
$$
there exist a sequence $\{\lambda_n\}_{n\geq 1}\subset\Lambda$ and a number $N\in\N$ such that for each $n\geq N$
\begin{equation}\label{f_lambda}
\gamma(\lambda_n)<\frac{\delta}{2}
\end{equation}
and
$$
\sup\limits_{u\in[0,1]}\rho_{\frac{1}{p}}\left(\pi_{\frac{1}{p}}f_n(\lambda_n(u),\cdot),\pi_{\frac{1}{p}}f(u,\cdot)\right)<\frac{\delta}{2}.
$$
Note that
$$
\sup\limits_{u\in[0,1]}\rho_{\frac{1}{k}}\left(\pi_{\frac{1}{k}}f_n(\lambda_n(u),\cdot),\pi_{\frac{1}{k}}f(u,\cdot)\right)<\frac{\delta}{2},\quad k\in[p],\ \ n\geq N.
$$
Hence for all $u\in[0,1]$
$$
\sum_{k=1}^p\frac{1}{2^k}\rho_{\frac{1}{k}}\left(\pi_{\frac{1}{k}}f_n(\lambda_n(u),\cdot),\pi_{\frac{1}{k}}f(u,\cdot)\right)<\frac{\delta}{2},\quad n\geq N.
$$
By the choice of $p$ and the latter inequality, for each $n\geq N$ and $u\in[0,1]$
$$
\rho(f_n(\lambda_n(u),\cdot),f(u,\cdot))<\delta.
$$
This and~\eqref{f_lambda} imply the inequality
$$
d_0(f_n,f)\leq\delta,\quad n\geq N.
$$
The lemma is proved.
\end{proof}

\begin{prp}\label{prop_tightness_apend}
A set $\{P_n\}_{n\geq 1}$ of probability measures in $D([0,1],C(0,T])$ is tight if and only if for every $\eps>0$ the set $\{P_n\circ\pi_{\eps}^{-1}\}_{n\geq 1}$ is tight in $D([0,1],C[\eps,T])$.
\end{prp}

\begin{proof}
The necessity follows from the continuity of the map $$\pi_{\eps}:D([0,1],C(0,T])\to D([0,1],C[\eps,T])$$ (see Lemma~\ref{lemma_convergence_in_D}~$(i)$). Let us prove the sufficiency. Let $\eps_q=\frac{1}{q}$ and $K_q$ be a compact set in $D([0,1],C[\eps_q,T])$, $q\in\N$. Show that
\begin{equation}\label{f_compact}
K=\bigcap\limits_{q=1}^{\infty}\pi_{\eps_q}^{-1}(K_q)
\end{equation}
is compact in $D([0,1],C(0,T])$. Since $K$ is closed and $D([0,1],C(0,T])$ is a metric space, to prove the compactness of $K$, it suffices to show that every sequence of elements of $K$ has a convergent subsequence. So, let $\{f_n\}_{n\geq 1}\subseteq K$. Note that $\{\pi_{\eps_q}f_n\}_{n\geq 1}\subseteq K_q$, for each $q\in\N$. Hence $\{\pi_{\eps_q}f_n\}_{n\geq 1}$ has a convergent subsequence.
By Cantor's diagonal argument, we can choose a sequence $\{n'\}\subseteq\N$ such that for all $q\in\N$ the sequence $\{\pi_{\eps_q}f_{n'}\}_{n'}$ tends to some function $g_q$ in $D([0,1],C[\eps_q,T])$. By Lemma~\ref{lemma_convergence_in_D}~$(ii)$, the sequence $\{f_{n'}\}_{n'}$ is convergent in $D([0,1],C(0,T])$. Thus, we obtain that $K$ is compact.

Let $\delta>0$ be fixed. By the tightness of $\{P_n\circ\pi_{\eps}^{-1}\}_{n\geq 1}$ we can take a compact set $K_q$ in $D([0,1],C[\eps_q,T])$ such that
$$
\inf\limits_{n\geq 1}P_n(\pi_{\eps}^{-1}(K_q))\geq 1-\frac{\delta}{2^q},\quad q\geq 1.
$$
Defining $K$ by~\eqref{f_compact} and using the standard argument we have the estimation
$$
P_n(K)\geq 1-\delta,\quad n\geq 1.
$$
It completes the proof of the proposition.
\end{proof}

\section{A special integral for a stochastic coalescing flow}\label{app_integral}

In this section we recall the construction of an integral with respect to a stochastic flow that was defined by A.~A.~Dorogovtsev and state the fact used in the proof of Theorem~\ref{theorem_local_time}. Let $\{y(u,t),\ u\in[0,1],\ t\in[0,T]\}$ be a random element of $D([0,1],C[0,T])$ which satisfies conditions $(C1)$--$(C5)$. For a set of points $\{u_k,\ k\in\N\}\subset [0,1]$ denote
\begin{align*}
\tau(u_1)&=T,\\
\tau(u_k)&=\inf\{t:\ \exists l\in[k-1]\ y(u_l,t)=y(u_k,t)\}\wedge T,\quad k=2,3,\ldots.
\end{align*}

\begin{lem}\label{lemma_conv_int_sum}
{\cob Let $\{u_k,\ k\in\N\}$ be dense in $[0,1]$. Then for each continuous bounded function $\varphi:\R\to\R$ and $t\in[0,T]$ the series $\sum_{k=1}^{+\infty}\int_0^{\tau(u_k)\wedge t}\varphi(y(u_k,s))ds$ converges almost surely. Moreover its sum is independent of the choice of $\{u_k,\ k\in\N\}$ and equals $\int_0^1\int_0^t\frac{\varphi(y(u,s))}{m(u,s)}dsdu$.}
\end{lem}

\begin{proof}
To prove the lemma first show that the series is absolutely convergent a.s. For each $n\in\N$ let
$$
S_n=\sum_{k=1}^{n}\left|\int_0^{\tau(u_k)\wedge t}\varphi(y(u_k,s))ds\right|
$$
and $\{u_k^n,\ k\in[n]\}$ be the set of $\{u_k,\ k\in[n]\}$ ordered by increasing. Set
\begin{align*}
\widetilde{\tau}(u_1^n)&=T,\\
\widetilde{\tau}(u_k^n)&=\inf\{t:\ y(u_{k-1}^n,t)=y(u_k^n,t)\}\wedge T,\quad k=2,\ldots,n.
\end{align*}
Then
$$
S_n=\sum_{k=1}^{n}\left|\int_0^{\widetilde{\tau}(u_k^n)\wedge t}\varphi(y(u_k^n,s))ds\right|
$$
and hence
$$
\E S_n=\sum_{k=1}^{n}\E\left|\int_0^{\widetilde{\tau}(u_k^n)\wedge t}\varphi(y(u_k^n,s))ds\right|\leq \|\varphi\|\sum_{k=1}^{n}\E(\widetilde{\tau}(u_k^n)\wedge t).
$$
By Lemma~7.1.1~\cite{Dorogovtsev:2007:en}, there exists a constant $C$ such that for all $n\in\N$
$
\E S_n\leq C.
$
So, the sequence $\{S_n\}_{n\geq 1}$ is convergent a.s. and consequently
$$
\sum_{k=1}^{+\infty}\int_0^{\tau(u_k)\wedge t}\varphi(y(u_k,s))ds
$$
converges a.s.

In order to prove the second part of the lemma set
\begin{align*}
\Omega'&=\{y\in D([0,1],C[0,T])\}\cap\{y(u,t)\leq y(v,t),\ u<v,\ t\in[0,T]\}\\
&\cap\{y\ \mbox{is a step function}\}\cap\{m(u,t)>0,\ u\in[0,1], t\in(0,T]\}\\
&\cap\{\mbox{for all}\ u,v\in[0,1]\ \mbox{if}\ y(u,t)=y(v,t)\ \mbox{then}\ y(u,t+\cdot)=y(v,t+\cdot)\}.
\end{align*}
By Theorem~\ref{theorem_main_result}, Lemmas~\ref{lemma_expect_of_y2},~\ref{lemma_step_funct}, Remark~\ref{remark_prop_y}, taking into account the relation
$
\E\int_0^1\int_0^t\frac{1}{m(u,s)}dsdu=\int_0^1\E(y(u,t)-u)^2du
$
and the monotonicity of $m(u,t),\ t\in[0,T]$, for all $u\in[0,1]$, we conclude that $\p\{\Omega'\}=1$.

Next, fix $\omega\in\Omega'$ and denote for $s\in(0,t]$
$$
A(u,s)=\{v\in[0,1]:\ \exists r\leq s\ y(u,r,\omega)=y(v,r,\omega)\}
$$
and
$
B(s)=\{\min A(u,s),\ u\in[0,1]\}.
$
From the choice of $\omega$ it easily {\cob follows that $B(s)$ is finite. Since $\{u_k,\ k\in\N\}$ is dense} in $[0,1]$ and $\inter A(u,s)\neq\emptyset$ (because $m(u,s,\omega)>0$), $u\in[0,1]$, there exists $N(s)\in\N$ such that $A(u,s)\cap\{u_k,\ k\in[N(s)]\}\neq\emptyset$, for all $u\in[0,1]$.

Let $\varepsilon<t$. Then for all $n\geq N(\varepsilon)$,
\begin{align*}
&\int_0^1\int_{\varepsilon}^t\frac{\varphi(y(u,s,\omega))}{m(u,s,\omega)}duds=\int_{\varepsilon}^t\sum_{v\in B(s)} \int_{A(v,s)}\frac{\varphi(y(u,s,\omega))}{m(u,s,\omega)}duds\\
&=\int_{\varepsilon}^t\sum_{v\in B(s)}\varphi(y(v,s,\omega))ds=\sum_{k=1}^n\int_{\varepsilon}^{(\varepsilon\vee\tau(u_k,\omega))\wedge t}\varphi(y(u_k,s,\omega))ds.
\end{align*}

So, we obtain that for all $\omega\in\Omega'$ and $\eps<t$
$$
\int_0^1\int_{\varepsilon}^t\frac{\varphi(y(u,s,\omega))}{m(u,s,\omega)}duds= \sum_{k=1}^{+\infty}\int_{\varepsilon}^{(\varepsilon\vee\tau(u_k,\omega))\wedge t}\varphi(y(u_k,s,\omega))ds.
$$
{\col
Next, if $\varphi$ is non-negative, then by the monotone convergence theorem
$$
\int_0^1\int_0^t\frac{\varphi(y(u,s,\omega))}{m(u,s,\omega)}duds= \sum_{k=1}^{+\infty}\int_0^{\tau(u_k,\omega)\wedge t}\varphi(y(u_k,s,\omega))ds,
$$
where the integral in the left hand side is finite a.s., by Lemma~\ref{lemma_estim_expect_char}. To get the statement for any continuous bounded function $\varphi$, it is needed to apply the obtained result to its positive and negative parts.
}
This completes the proof.
\end{proof}

Lemma~\ref{lemma_conv_int_sum} implies that the sum $\sum_{k=1}^{+\infty}\int_0^{\tau(u_k)\wedge t}\varphi(y(u_k,s))ds$ does not depend on the choice of a dense set $\{u_k,\ k\in\N\}$. So we set
$$
\int_0^1\int_0^{\tau(u)\wedge t}\varphi(y(u,s))ds=\sum_{k=1}^{+\infty}\int_0^{\tau(u_k)\wedge t}\varphi(y(u_k,s))ds
$$

\begin{cor}\label{lemma_integral}
Let $\varphi$ be a bounded continuous function. Then a.s.
$$
\int_0^1\int_0^t\frac{\varphi(y(u,s))}{m(u,s)}dsdu=\int_0^1\int_0^{\tau(u)\wedge t}\varphi(y(u,s))ds.
$$
\end{cor}

\textbf{Acknowledgements.} The author is thankful to Prof.~Dr.~A.~A.~Dorogovtsev for the statement of the problem and for help during this work. The author also is grateful to Max von Renesse and Florent Barret for useful discussions and suggestions. {\col Special thanks go to the anonymous referee for helpful suggestions and comments.}


\end{document}